\documentclass{amsart}
\usepackage [latin1]{inputenc}
\usepackage{amsmath,amssymb,amsbsy,amsfonts,amsthm,latexsym,
                       amsopn,amstext,amsxtra,euscript,amscd}
                       \usepackage{color}
\usepackage{array}
\usepackage{ifthen}
\usepackage{url}

\newtheorem{Thm}{Theorem}[section]
\newtheorem{Lem}[Thm]{Lemma}

\newtheorem{Conj}[Thm]{Conjecture}

\theoremstyle{definition}

\theoremstyle{remark}

\newcommand{\klammern}[4][]%
{\ifthenelse{\equal{#1}{}}{\left#2}{\csname#1\endcsname#2}%
#4\ifthenelse{\equal{#1}{}}{\right#3}{\csname#1\endcsname#3}}

\renewcommand{\epsilon}{\DONOTUSE}  

\renewcommand{\vartheta}{\DONOTUSE} 

\begin{document}

\title[The Explicit Formula and  Parity for Some Generalized Euler Functions ] {The Explicit  Formula and  Parity for Some Generalized Euler Functions}


\author[Shichun Yang]{Shichun Yang}
\address{School of mathematics,
ABa Teachers University,
Wenchuan 623000, Sichuan, P. R. China}
\email{ysc1020@sina.com}

\author[Qunying Liao]{Qunying Liao}
\address{School of Mathematical Sciences, Sichuan Normal University, Chengdu 610066, Sichuan, P. R. China}
 \email{qunyingliao@sicun.edu.cn}

\author[Shan Du]{Shan Du}
\address{School of Mathematical Sciences, Sichuan Normal University, Chengdu 610066, Sichuan, P. R. China}
 \email{dubingha@qq.com}

\author[Huili Wang]{Huili Wang}
\address{School of Mathematical Sciences, Sichuan Normal University, Chengdu 610066, Sichuan, P. R. China}
 \email{813870971@qq.com}



\date{\today}


\begin{abstract} Based on elementary methods and techniques, the explicit formula for the generalized Euler function $\varphi_{e}(n)(e=8,12)$ is given, and then a sufficient and necessary condition for $\varphi_{8}(n)$ or $\varphi_{12}(n)$ to be odd is obtained, respectively.
\end{abstract} \maketitle

\section{Introduction}\label{sec:1}

Let  $\mathbb{Z}$ and $\mathbb{P}$ denote the set of  integers and primes, respectively. In order to generalize the Lehmer's congruence (see [6]) from modulo prime squares to be modulo
integer squares, Cai  [3] defined the following generalized Euler function for a positive integer $n$ related to a given positive integer $e$:
\begin{eqnarray}\label{eq:1}
\varphi_{e}(n)=\sum_{i=1, \,\gcd(i,n)=1} ^{[\frac{n}{e}]}1, \nonumber
\end{eqnarray}
where $[x]$ is the greatest integer not more than $x$, i.e., $\varphi_{e}(n)$ is the number of positive integers
not greater than $[\frac{n}{e}]$ and prime to $n$.
It is easy to verify that $\varphi_{1}(n)=\varphi(n)$ is just the Euler function of $n$, $\varphi_{2}(n)=\frac{1}{2}\varphi(n)$, and
\begin{eqnarray}\label{eq:M}
\varphi_{e}(n)=\sum_{d|n}\mu\Big(\frac{n}{d}\Big)\Big[\frac{d}{e}\Big],
\end{eqnarray}
where $\mu(n)$ is the M\"{o}bius function. There are some good results for the generalized Euler function and its applications, especially, concerning on
$\varphi_{e}(n)(e=2,3,4,6)$, which can be seen in [1], [2], [5] and [9].

In 2013, Cai, Shen and Hu [4] gave the explicit formula for $\varphi_{3}(n)$, and obtained a criterion on the parity for $\varphi_{2}(n)$ or $\varphi_{3}(n)$, respectively. In [10], the authors gave the explicit formulae for $\varphi_{4}(n)$ and $\varphi_{6}(n)$, and then
obtained some sufficient and necessary conditions for that  $\varphi_{e}(n)$ or $\varphi_{e}(n+1)$ is odd or even, respectively.

Recently, Wang and Liao [11] gave the formula for $\varphi_{5}(n)$ in some special cases
 and then obtained some sufficient conditions for that $\varphi_{5}(n)$ is even. Liao and Luo [7] gave a computing formula for
 $\varphi_{e}(n)$ $(e=p, p^{2}, pq)$, where $p$ and $q$ are distinct primes, and $n$ satisfies some certain conditions. Liao [8]
obtained the explicit formula for a special class of  generalized Euler functions. However, the explicit formula for $\varphi_{e}(n)(e\neq3, 4, 6)$ is not obtained in the general case.

In this paper, basing on the methods and techniques given in [4], [10] and [7], we study the explicit formula and the parity for $\varphi_{e}(n)(e=8,12)$, obtain the corresponding computing formula, and then give a sufficient and necessary condition for that $\varphi_{e}(n)(e=8,12)$ is odd or even, respectively.

For the convenience, the whole paper denotes $\Omega(n)$ and $\omega(n)$ to be the number of prime factors and distinct prime factors of a positive integer $n$, respectively. And for $k$ primes $p_{1}, \ldots, p_{k}$, set $\mathbb{P}_{k}=\{p_{1}, \ldots, p_{k}\}$,
$$
R_{\mathbb{P}_{k}}=\{r_{i}\mid p_{i}\equiv r_{i}(\mathrm{mod}\, 8), 0\leq r_{i}\leq7, p_{i}\in \mathbb{P}_{k}, 1\leq i\leq k \},
$$
and
$$R_{\mathbb{P}_{k}}'=\{r_{i}\mid p_{i}\equiv r_{i}(\mathrm{mod}\, 12), 0\leq r_{i}\leq11, p_{i}\in\mathbb{P}_{k}, 1\leq i\leq k \}.$$

\section{Preliminaries}\label{sec:2}
In this section, we first give the following Lemmas 2.1-2.2, which are necessary for the calculations of both $[\frac{m}{8}]$ and $[\frac{m}{12}]$.
\begin{Lem}\label{lem:2.1}\quad  For any odd positive integer $m$, we have
\begin{eqnarray}\label{eq:[m/8]}
\Big[\frac{m}{8}\Big]=\frac{1}{8}\Big(m-4+2\Big(\frac{-2}{m}\Big)+\Big(\frac{-1}{m}\Big)\Big).
\end{eqnarray}
Furthermore, if $\gcd(m, 6) = 1$, then we have
\begin{eqnarray}\label{eq:[m/12]}
\Big[\frac{m}{12}\Big]=\frac{1}{12}\Big(m-6+3\Big(\frac{-1}{m}\Big)+2\Big(\frac{-3}{m}\Big)\Big),
\end{eqnarray}
where $(\frac{a}{m})$  is the Jacobi symbol.
\end{Lem}
\begin{proof}
For any odd positive integer $m$, by properties of the Jacobi symbol, we have
\begin{eqnarray}\label{eq:[m]}
\Big(\frac{-1}{m}\Big)=\begin{cases}
1,  & \mbox{$m\equiv1(\mathrm{mod}\,4)$},\\
-1,  & \mbox{$m\equiv3(\mathrm{mod}\,4)$},
 \end{cases} \;\; \mbox{and} \;\;\Big(\frac{-2}{m}\Big)=\begin{cases}
1,  & \mbox{$m\equiv1, 3(\mathrm{mod}\,8)$},\\
-1,  & \mbox{$m\equiv5, 7(\mathrm{mod}\,8)$}.
 \end{cases}  \nonumber
\end{eqnarray}
Thus from $m\equiv1(\mathrm{mod}\,8)$, we can get $\frac{1}{8}(m-4+2(\frac{-2}{m})+(\frac{-1}{m}))=\frac{1}{8}(m-1)$ and $[\frac{m}{8}]=\frac{1}{8}(m-1)$, namely, (2) is true. Similarly, if $m\equiv3, 5, 7(\mathrm{mod}\,8)$, by computing directly, (2) holds.

Furthermore, if $\gcd(m, 6) = 1$, then by the properties for the Jacobi symbol and the quadratic reciprocity law, we have
\begin{eqnarray}\label{eq:A2}
\Big(\frac{-3}{m}\Big)=\Big(\frac{-1}{m}\Big)\Big(\frac{m}{3}\Big)(-1)^{\frac{1}{4}(3-1)(m-1)}=\begin{cases}
1,  & \mbox{$m\equiv1, 7(\mathrm{mod}\,12)$},\\
-1,  & \mbox{$m\equiv5, 11(\mathrm{mod}\,12)$}.
 \end{cases}\nonumber
\end{eqnarray}
Thus by $m\equiv1(\mathrm{mod}\,12)$, we have $\frac{1}{12}(m-6+3(\frac{-1}{m})+2(\frac{-3}{m}))=\frac{1}{12}(m-1)=[\frac{m}{12}]$, i.e., (3) is true. Similarly, if $m\equiv5, 7, 11(\mathrm{mod}\,12)$, one can get (3) also.
\end{proof}

Now, we  give a property for the M\"{o}bius function, which unifies the cases of Lemma 1.5 in [4] and Lemmas 1.4-1.5 in [10].

\begin{Lem}\label{lem:2.2}\quad Let $a$ be a nonzero integer, $p_1,\ldots,p_k$ be distinct odd primes, and $\alpha_1,\ldots,\alpha_k$ be positive integers. Suppose that $n=\prod_{i=1}^{k} p_{i}^{\alpha_{i}}$ and $\gcd(p_{i}, a)=1 \,(1 \leq i \leq k)$, then
\begin{eqnarray}\label{eq:Mh(n)}
\sum_{d|n}\mu\Big(\frac{n}{d}\Big)\Big(\frac{a}{d}\Big)=
\prod_{i=1}^{k}\Big(\Big(\frac{a}{p_{i}}\Big)^{\alpha_{i}}-\Big(\frac{a}{p_{i}}\Big)^{\alpha_{i}-1}\Big).
\end{eqnarray}
\end{Lem}
\begin{proof}
For a given integer $x$, set $f_x(m)=\sum_{d|m}\mu(\frac{m}{d})(\frac{x}{d}).$

First, if $m=p^{\alpha}$, where $p$ is an odd prime and $\alpha$ is a positive integer. Then by the definition of the M\"{o}bius function, we have
\begin{eqnarray}\label{eq:2.1}
f_a(m)=\mu(1)\Big(\frac{a}{p^{\alpha}}\Big)+\mu(p)\Big(\frac{a}{p^{\alpha-1}}\Big)
=\Big(\frac{a}{p}\Big)^{\alpha}-\Big(\frac{a}{p}\Big)^{\alpha-1}.\nonumber
\end{eqnarray}

Secondly, if $m=m_{1} p^{\alpha}$, where $\alpha$ is a positive integer,  $p$ is an odd prime with $\gcd(m_{1}, p)=1$, and $m_{1}$ is an odd positive integer. Then we have
\begin{eqnarray}\label{eq:2.1}
f_a(m)&=&\sum_{d|m_{1}}\mu\Big(\frac{m_{1}}{d }\Big)\Big(\frac{a}{d p^{\alpha}}\Big)+\sum_{d|m_{1}}\mu(p)\mu\Big(\frac{m_{1}}{d }\Big)\Big(\frac{a}{d p^{\alpha-1}}\Big)\nonumber\\
&=&\Big(\frac{a}{p}\Big)^{\alpha}\sum_{d|m_{1}}\mu\Big(\frac{m_{1}}{d }\Big)\Big(\frac{a}{d}\Big)-\Big(\frac{a}{p}\Big)^{\alpha-1}\sum_{d|m_{1}}\mu\Big(\frac{m_{1}}{d }\Big)\Big(\frac{a}{d}\Big).\nonumber\\
&=&\Big(\Big(\frac{a}{p}\Big)^{\alpha}-\Big(\frac{a}{p}\Big)^{\alpha-1}\Big)f_a(m_{1}).\nonumber
\end{eqnarray}
This means that $f_a(m)$ is a multiplicative function. Now denote $p^{\alpha}\|n$ to be the case both $p^{\alpha}\mid n$ and $p^{\alpha+1}\nmid n$,  then we can get
\begin{eqnarray}\label{eq:2.1}
f_a(n)=\prod_{p^{\alpha}\| n}\Big(\Big(\frac{a}{p}\Big)^{\alpha}-\Big(\frac{a}{p}\Big)^{\alpha-1}\Big)=
\prod_{i=1}^{k}\Big(\Big(\frac{a}{p_{i}}\Big)^{\alpha_{i}}-\Big(\frac{a}{p_{i}}\Big)^{\alpha_{i}-1}\Big).\nonumber
\end{eqnarray}

This completes the proof of Lemma 2.2.
\end{proof}

The following lemmas are necessary for proving our main results.

\begin{Lem}\label{lem:2.3}[4]\quad
Let $p_1,\ldots, p_k$ be distinct primes, and $\alpha,\alpha_1,\ldots,\alpha_k$ be nonnegative integers. If $n=3^{\alpha}\,\prod_{i=1}^{k} p_{i}^{\alpha_{i}}>3$ and $\gcd(p_{i}, 3) = 1 \,(1 \leq i \leq k)$, then
\begin{eqnarray}\label{eq:AB3-0}
\varphi_{3}(n)=\begin{cases}
\frac{1}{3}\,\varphi(n)+\frac{1}{3}\,(-1)^{\Omega(n)}2^{\omega(n)-\alpha-1},  & \mbox{if $\alpha=0$ or $1$, $p_{i}\equiv2(\mathrm{mod}\,3)$},\\
\frac{1}{3}\,\varphi(n),  & \mbox{otherwise}.
 \end{cases}\nonumber
\end{eqnarray}
\end{Lem}

\begin{Lem}\label{lem:2.4}[10]\quad Let $p_1,\ldots, p_k$ be distinct odd primes, and $\alpha,\alpha_1,\ldots,\alpha_k$ be nonnegative integers.
If $n =2^{\alpha}\,\prod_{i=1}^{k} p_{i}^{\alpha_{i}}>4$, then
\begin{eqnarray}\label{eq:AB3-0}
\varphi_{4}(n)=\begin{cases}
\frac{1}{4}\,\varphi(n)+\frac{1}{4}\,(-1)^{\Omega(n)}2^{\omega(n)-\alpha},  & \mbox{if $\alpha=0$ or $1$, $p_{i}\equiv3(\mathrm{mod}\,4)$},\\
\frac{1}{4}\,\varphi(n),  & \mbox{otherwise}.
 \end{cases}\nonumber
\end{eqnarray}
\end{Lem}

\begin{Lem}\label{lem:2.5}[10]\quad  Let $p_1,\ldots, p_k$ be distinct primes, and $\alpha,\beta,\alpha_1,\ldots,\alpha_k$ be nonnegative integers.
If  $n =2^{\alpha}\,3^{\beta}\,\prod_{i=1}^{k} p_{i}^{\alpha_{i}}>6$ and $\gcd(p_{i}, 6) = 1 \,(1 \leq i \leq k)$, then
\begin{eqnarray}\label{eq:A2}
\varphi_{6}(n)=\begin{cases}
\frac{1}{6}\,\varphi(n)+\frac{1}{6}\,(-1)^{\Omega(n)}\,2 \,^{\omega(n)+1-\beta},  & \mbox{if $\alpha=0$ and $\beta=0$ or $1$, $p_{i}\equiv5(\mathrm{mod}\,6)$},\\
\frac{1}{6}\,\varphi(n)+\frac{1}{6}\,(-1)^{\Omega(n)}\,2 \,^{\omega(n)-1-\beta},  & \mbox{if $\alpha=1$ and $\beta=0$ or $1$, $p_{i}\equiv5(\mathrm{mod}\,6)$},\\
\frac{1}{6}\,\varphi(n)-\frac{1}{6}\,(-1)^{\Omega(n)}\,2 \,^{\omega(n)-\beta},  & \mbox{if $\alpha\geq2$ and $\beta=0$ or $1$, $p_{i}\equiv5(\mathrm{mod}\,6)$},\\
\frac{1}{6}\,\varphi(n),  & \mbox{otherwise}.
 \end{cases}\nonumber
\end{eqnarray}
\end{Lem}

\begin{Lem}\label{lem:2.6}[7]\quad  Let $p_1,\ldots, p_k$ be distinct primes, and $\alpha_1,\ldots,\alpha_k$ be positive integers.
If $n =\prod_{i=1}^{k} p_{i}^{\alpha_{i}}$ and $e=\prod_{i=1}^{k} p_{i}^{\beta_{i}}$ with $0\leq\beta_{i}\leq\alpha_{i}-1(1 \leq i \leq k)$,
then
\begin{eqnarray}\label{eq:n/e}
\varphi_{e}(n)=\frac{1}{e}\,\varphi(n).
\end{eqnarray}
\end{Lem}

\section{The Explicit  Formula  For $\varphi_{8}(n)$}\label{sec:3}

First, for a fixed positive integer $\alpha$ and $n=2^{\alpha}$, by Lemma 2.6 we can obtain
\begin{eqnarray}\label{eq:A0}
\varphi_{8}(2^{\alpha})=\begin{cases}
0,  & \mbox{if $\alpha=1, 2$},\\
1,  & \mbox{if $\alpha=3$},\\
2^{\alpha-4},  & \mbox{if  $\alpha\geq 4$}.
 \end{cases}
\end{eqnarray}

Next, we consider the case $n=2^{\alpha}n_{1}$, where $n_{1}>1$ is an odd integer.  We have the following
\begin{Thm}\label{thm:3.1}\quad Suppose that $\alpha$ is a non-negative integer, $p_{1},\ldots, p_k$ are distinct odd primes, and $n=2^{\alpha}\,\prod_{i=1}^{k} p_{i}^{\alpha_{i}}>8.$  Then we have
\begin{eqnarray}\label{eq:A1}
\varphi_{8}(n)=\begin{cases}
\frac{1}{8}\,\varphi(n)+\frac{1}{4}\,(-1)^{\Omega(n)}\,2 \,^{\omega(n)-\alpha},\\
  \;\;\;\;\;\;\;\;\;\;\;\;\;\;\;\;\;\;\;\;\;\;\;\;\;\;\;\;\;\;\;\;\;\;\;\; \mbox{if $\alpha=0, 1, \mbox{and}\ R_{\mathbb{P}_{k}}=\{5, 7\}, \{5\}$};\\
\frac{1}{8}\,\varphi(n)+\frac{1}{8}\,(-1)^{\Omega(n)-[\frac{\alpha+1}{2}]}\,2 \,^{\omega(n)-\frac{1}{2}(1-(-1)^{\alpha})},\\
  \;\;\;\;\;\;\;\;\;\;\;\;\;\;\;\;\;\;\;\;\;\;\;\;\;\;\;\;\;\;\;\;\;\;\;\mbox{if $\alpha=0, 1, 2, \mbox{and}\ R_{\mathbb{P}_{k}}=\{3, 7\}, \{3\}$};\\
\frac{1}{8}\,\varphi(n)+\frac{1}{8}\,(-1)^{\Omega(n)-[\frac{\alpha}{2}]}\,2 \,^{\omega(n)-\frac{1}{2}(1-(-1)^{\alpha})}\\
  \;\;\;\;\;\;\;\;\;\;\;\;\;\;\;\;\;\;\;\;\;\;\;\;\;\;\;\;\;\;\;\;\;+\frac{1-[\frac{\alpha+1}{2}]}{4}\,(-1)^{\Omega(n)}\,2^{\omega(n)},\\
\;\;\;\;\;\;\;\;\;\;\;\;\; \;\;\;\;\;\;\;\;\;\;\;\;\;\;\;\;\;\;\;\;\;\; \mbox{if $\alpha=0, 1, 2$, \mbox{and}\ $R_{\mathbb{P}_{k}}=\{7\}$};\\
\frac{1}{8}\,\varphi(n), \;\;\;\;\;\;\;\;\;\;\;\; \;\;\;\;\;\;\;\; \;\; \mbox{otherwise}.
 \end{cases}
\end{eqnarray}
\end{Thm}

\begin{proof} \quad For $n=2^{\alpha}\,\prod_{i=1}^{k} p_{i}^{\alpha_{i}}>8$, set $n_1 = \prod_{i=1}^{k} p_{i}^{\alpha_{i}}$, then $\gcd(n_{1}, 2)= 1$. There are 4 cases as follows.

\textbf{Case 1}. \quad For $\alpha= 0$, i.e., $n=n_{1} > 8$. By (1)-(2) and Lemmas 2.1-2.2, we have
\begin{eqnarray}\label{eq:n/8}
\varphi_{8}(n)&= &\sum_{d\,|\,n_{1}}\mu\Big(\frac{n_{1}}{d}\Big)\Big[\frac{d}{8}\Big]
=\frac{1}{8}\sum_{d\,|\,n_{1}}\mu\Big(\frac{n_{1}}{d}\Big)\Big(d-4+2\Big(\frac{-2}{d}\Big)+\Big(\frac{-1}{d}\Big)\Big)\\
&=&\frac{1}{8}\sum_{d\,|\,n_{1}}\mu\Big(\frac{n_{1}}{d}\Big)d-\frac{1}{2}\sum_{d\,|\,n_{1}}\mu\Big(\frac{n_{1}}{d}\Big)
+\frac{1}{4}\sum_{d\,|\,n_{1}}\mu\Big(\frac{n_{1}}{d}\Big)\Big(\frac{-2}{d}\Big)\nonumber\\
&&+\;\frac{1}{8}\sum_{d\,|\,n_{1}}\mu\Big(\frac{n_{1}}{d}\Big)\Big(\frac{-1}{d}\Big)\nonumber\\
&=&\frac{1}{8}\,\varphi(n_{1})+\frac{1}{4}\,\prod_{i=1}^{k}\Big(\Big(\frac{-2}{p_{i}}\Big)^{\alpha_{i}}-\Big(\frac{-2}{p_{i}}\Big)^{\alpha_{i}-1}\Big)
\nonumber\\
&&+\,\frac{1}{8}\,\prod_{i=1}^{k}\Big(\Big(\frac{-1}{p_{i}}\Big)^{\alpha_{i}}-\Big(\frac{-1}{p_{i}}\Big)^{\alpha_{i}-1}\Big).\nonumber
\end{eqnarray}

If $1\in R_{\mathbb{P}_{k}}$, i.e.,  there exists an $i(1\leq i\leq k)$ such that $p_{i}\equiv1 (\mathrm{mod}\, 8)$, and so $(\frac{-2}{p_{i}}) = (\frac{-1}{p_{i}}) =1$. Now by (8) we have
\begin{eqnarray}\label{eq:8n}
\varphi_{8}(n)=\frac{1}{8}\,\varphi(n_{1})=\frac{1}{8}\,\varphi(n).
\end{eqnarray}

If $\{3,5\}\subseteq R_{\mathbb{P}_{k}}$, i.e., there exist $i\neq j$ such that $p_{i}\equiv3(\mathrm{mod}\,8)$ and $p_{j}\equiv5(\mathrm{mod}\,8)$, which means that $(\frac{-2}{p_{i}}) = (\frac{-1}{p_{j}}) =1$. Thus by (8) we also have
\begin{eqnarray}\label{eq:8n}
\varphi_{8}(n)=\frac{1}{8}\,\varphi(n_{1})=\frac{1}{8}\,\varphi(n).\nonumber
\end{eqnarray}

If $R_{\mathbb{P}_{k}}=\{5, 7\}$ or $\{5\}$, i.e., for any $p\in\mathbb{P}_{k}$,  we have $p\equiv5, 7(\mathrm{mod}\,8)$ or $p\equiv5(\mathrm{mod}\,8)$, respectively. This means that there exists a prime $p$ such that $(\frac{-2}{p})=-1$ and $(\frac{-1}{p})=1$. Thus by (8) we can obtain
\begin{eqnarray}\label{eq:8n1}
\varphi_{8}(n)=\frac{1}{8}\,\varphi(n_{1})+\frac{1}{4}\,\prod_{i=1}^{k}\Big(2\cdot(-1)^{\alpha_{i}}\Big)
=\frac{1}{8}\,\varphi(n)+\frac{1}{4}\,(-1)^{\Omega(n)}\,2 \,^{\omega(n)}.
\end{eqnarray}

If $R_{\mathbb{P}_{k}}=\{3, 7\}$ or $\{3\}$, i.e., for any $p\in\mathbb{P}_{k},  p\equiv3, 7(\mathrm{mod}\,8)$ or $p\equiv3(\mathrm{mod}\,8)$, respectively. This implies that  for any $p\in\mathbb{P}_{k}, (\frac{-1}{p})=-1$, and there exists a prime $p'\in\mathbb{P}_{k}$ such that $p'\equiv3(\mathrm{mod}\,8)$ and then $(\frac{-2}{p' })=1$. Thus by (8) we have
\begin{eqnarray}\label{eq:8n2}
\varphi_{8}(n)=\frac{1}{8}\,\varphi(n_{1})+\frac{1}{8}\,\prod_{i=1}^{k}\Big(2\cdot(-1)^{\alpha_{i}}\Big)
=\frac{1}{8}\,\varphi(n)+\frac{1}{8}\,(-1)^{\Omega(n)}\,2 \,^{\omega(n)}.
\end{eqnarray}

If $R_{\mathbb{P}_{k}}=\{7\}$, i.e.,  for any $p\in\mathbb{P}_{k}, p\equiv7 (\mathrm{mod}\, 8)$, and so $(\frac{-2}{p})=(\frac{-1}{p})=-1$. Thus by (8) we have
\begin{eqnarray}\label{eq:8n3}
\varphi_{8}(n)=\frac{1}{8}\,\varphi(n_{1})+\frac{3}{8}\,\prod_{i=1}^{k}\Big(2\cdot(-1)^{\alpha_{i}}\Big)
=\frac{1}{8}\,\varphi(n)+\frac{3}{8}\,(-1)^{\Omega(n)}\,2\,^{\omega(n)}.
\end{eqnarray}

Now from (10)-(12) we know that Theorem 3.1 is true.

\textbf{Case 2}.\quad For $\alpha= 1$, i.e., $n=2n_{1}>8$. Then  from the definition we have
\begin{eqnarray}\label{eq:2.2}
\varphi_{8}(n)
&=&\sum_{d|n_{1}}\mu\Big(\frac{2n_{1}}{d}\Big)\Big[\frac{d}{8}\Big]+\sum_{d|n_{1}}\mu\Big(\frac{2n_{1}}{2d}\Big)\Big[\frac{2d}{8}\Big]\nonumber\\
&=&-\varphi_{8}(n_{1})+\varphi_{4}(n_{1}).\nonumber
\end{eqnarray}
Now by Lemma 2.4 and the proof for Case 1, we can  get
\begin{eqnarray}\label{eq:8n4}
\varphi_{8}(n)=\begin{cases}
\frac{1}{8}\,\varphi(n)+\frac{1}{4}\,(-1)^{\Omega(n)}\,2 \,^{\omega(n)-1},  & \mbox{if $R_{\mathbb{P}_{k}}=\{5, 7\}, \{5\}$};\\
\frac{1}{8}\,\varphi(n)+\frac{1}{8}\,(-1)^{\Omega(n)-1}\,2 \,^{\omega(n)-1},  & \mbox{if $R_{\mathbb{P}_{k}}=\{3, 7\}, \{3\}$};\\
\frac{1}{8}\,\varphi(n)+\frac{1}{8}\,(-1)^{\Omega(n)}\,2 \,^{\omega(n)-1},  & \mbox{if  $R_{\mathbb{P}_{k}}=\{7\}$};\\
\frac{1}{8}\,\varphi(n),  & \mbox{otherwise}.
 \end{cases}
\end{eqnarray}
This means that Theorem 3.1 is true in this case.

\textbf{Case 3}. \quad For $\alpha= 2$, i.e., $n=4n_{1}>8$. Then from the definition we have
 \begin{eqnarray}\label{eq:2.2}
\varphi_{8}(n)
&=&\sum_{d\,|\,n_{1}}\mu\Big(\frac{4n_{1}}{d}\Big)\Big[\frac{d}{8}\Big]+
\sum_{d\,|\,n_{1}}\mu\Big(\frac{4n_{1}}{2d}\Big)\Big[\frac{2d}{8}\Big]+
\sum_{d\,|\,n_{1}}\mu\Big(\frac{4n_{1}}{4d}\Big)\Big[\frac{4d}{8}\Big]\nonumber\\
&=&\sum_{d\,|\,n_{1}}\mu\Big(\frac{2n_{1}}{d}\Big)\Big[\frac{d}{4}\Big]+
\sum_{d\,|\,n_{1}}\mu\Big(\frac{n_{1}}{d}\Big)\Big[\frac{d}{2}\Big]\nonumber\\
&=&\varphi_{2}(n_{1})-\varphi_{4}(n_{1})=\frac{1}{2}\varphi(n_{1})-\varphi_{4}(n_{1}).\nonumber
\end{eqnarray}
Now from Lemma 2.4 and the proof for Case 1, we can also get
\begin{eqnarray}\label{eq:8n5}
\varphi_{8}(n)=\begin{cases}
\frac{1}{8}\,\varphi(n), & \mbox{if $R_{\mathbb{P}_{k}}=\{5, 7\}, \{5\}$},\\
\frac{1}{8}\,\varphi(n)+\frac{1}{8}\,(-1)^{\Omega(n)-1}\,2 \,^{\omega(n)},  & \mbox{if $R_{\mathbb{P}_{k}}=\{3, 7\}, \{3\}$};\\
\frac{1}{8}\,\varphi(n)+\frac{1}{8}\,(-1)^{\Omega(n)-1}\,2 \,^{\omega(n)},  & \mbox{if  $R_{\mathbb{P}_{k}}=\{7\}$};\\
\frac{1}{8}\,\varphi(n),  & \mbox{otherwise}.
 \end{cases}
\end{eqnarray}
This means that Theorem 3.1 holds in this case.

\textbf{Case 4}. \quad For $\alpha\geq 3$. Note that $\mu(2^{\gamma})=0$ for any positive integer $\gamma\geq 2$, thus by (1) and Lemma  2.4 we have
 \begin{eqnarray}\label{eq:8n6}
\varphi_{8}(n)
=\sum_{d\,|\,n_{1}}\mu\Big(\frac{2n_{1}}{d}\Big)\Big[\frac{2^{\alpha-1}d}{8}\Big]+
\sum_{d\,|\,n_{1}}\mu\Big(\frac{n_{1}}{d}\Big)\Big[\frac{2^{\alpha}d}{8}\Big].
\end{eqnarray}
If $\alpha=3$, then
$$\varphi_{8}(n)=-\sum_{d\,|\,n_{1}}\mu\Big(\frac{n_{1}}{d}\Big)\Big[\frac{d}{2}\Big]
+\sum_{d\,|\,n_{1}}\mu\Big(\frac{n_{1}}{d}\Big)d
=-\frac{1}{2}\varphi(n_1)+\varphi(n_1)=\frac{1}{2}\varphi(n_1)=\frac{1}{8}\,\varphi(n)$$
If $\alpha\geq4$, then
$\varphi_{8}(n)=-2^{\alpha-4}\varphi(n_{1})+2^{\alpha-3}\varphi(n_{1})
=2^{\alpha-4}\varphi(n_{1})=\frac{1}{8}\,\varphi(n)$,
which means that Theorem 3.1  holds also.

Form the above, we complete the proof of Theorem 3.1.
\end{proof}

\section{The Explicit  Formula  For $\varphi_{12}(n)$ }\label{sec:4}

In this section, we give the explicit formula for $\varphi_{12}(n)$. Obviously, $\varphi_{12}(n)=0$ when $n< 12$, and $\varphi_{12}(n)=1$ when $n=12$ or $24$, then we consider about $n> 12$ and $n\neq24$.

\begin{Thm}\label{thm:4.1}\quad Let $\alpha$ and $\beta$ be non-negative integers. If $n=2^{\alpha}\,3^{\beta}>12$ and $n\neq 24$, then
\begin{eqnarray}\label{eq:A2}
\;\;\;\; \;\;\;\;\varphi_{12}(n)=\begin{cases}
\frac{1}{2}\,(3^{\beta-2}-(-1)^{\alpha+\beta}),  &\mbox{if $\alpha=0,\,or\,\alpha\geq1,\beta\geq2$};\\
2^{\alpha-2}\cdot 3^{\beta-2}, &\mbox{if $\alpha\geq 2,\beta\geq 2$};\\
\frac{1}{3}\,(2^{\alpha+\beta-3}+(-1)^{\alpha+\beta}), &\mbox{if $\alpha\geq4,\beta=0,1$}.
 \end{cases}
\end{eqnarray}
\end{Thm}
\begin{proof}\quad
(1)For the case $\alpha=0$, i.e., $n=3^{\beta}>12$, and so $\beta\geq3$, then we have
\begin{eqnarray}\label{eq:2.1}
\varphi_{12}(3^{\beta})&=&\sum_{d\,|\,3^{\beta}}\mu\Big(\frac{3^{\beta}}{d}\Big)\Big[\frac{d}{12}\Big]
=\Big[\frac{3^{\beta}}{12}\Big]-\Big[\frac{3^{\beta-1}}{12}\Big]
=\Big[\frac{3^{\beta-1}}{4}\Big]-\Big[\frac{3^{\beta-2}}{4}\Big]\nonumber\\
&=& \frac{1}{4}\big(3^{\beta-1}-2+(-1)^{\beta-1}\big)-\frac{1}{4}\big(3^{\beta-2}-2+(-1)^{\beta-2}\big)  \nonumber \\
&=& \frac{1}{2}\big(3^{\beta-2}-(-1)^{\beta}\big).\nonumber
\end{eqnarray}

(2) For the case $\alpha=1$, i.e., $n=2\cdot3^{\beta}>12$, and so $\beta\geq2$, by Lemma 2.5,
\begin{eqnarray}\label{eq:2.1}
\varphi_{12}(2\cdot3^{\beta})&=&\sum_{d\,|\,3^{\beta}}\mu\Big(\frac{2\cdot3^{\beta}}{d}\Big)\Big[\frac{d}{12}\Big]
+\sum_{d\,|\,3^{\beta}}\mu\Big(\frac{2\cdot3^{\beta}}{2d}\Big)\Big[\frac{2d}{12}\Big]\nonumber\\
&=&-\varphi_{12}(3^{\beta})+\varphi_{6}(3^{\beta})
=-\frac{1}{12}\varphi(3^{\beta})+\frac{1}{2}\,(-1)^{\beta}+\frac{1}{6}\varphi(3^{\beta})\nonumber\\
&=&\frac{1}{2}\big(3^{\beta-2}-(-1)^{\beta+1}\big).\nonumber
\end{eqnarray}

(3) For the case $\alpha=2$, i.e.,$n=4\cdot 3^{\beta}>12$, and so $\beta\geq2$, then we have
\begin{eqnarray}\label{eq:2.1}
\varphi_{12}(4\cdot3^{\beta})&=&
\sum_{d\,|\,4\cdot3^{\beta}}\mu\Big(\frac{4\cdot3^{\beta}}{d}\Big)\Big[\frac{d}{12}\Big]\nonumber\\
&=&\mu(1)\Big[\frac{4\cdot3^{\beta}}{12}\Big]+
\mu(2)\Big[\frac{2\cdot3^{\beta}}{12}\Big]
+\mu(3)\Big[\frac{4\cdot3^{\beta-1}}{12}\Big]+\mu(6)\Big[\frac{2\cdot3^{\beta-1}}{12}\Big]\nonumber\\
&=&3^{\beta-1}-\Big[\frac{3^{\beta-1}}{2}\Big]
-3^{\beta-2}+\Big[\frac{3^{\beta-2}}{2}\Big]\nonumber\\
&=&3^{\beta-2}.\nonumber
\end{eqnarray}

(4) For the case $\alpha=3$, i.e.,$n=8\cdot 3^{\beta}>12$ and $n\neq24$, and so $\beta\geq2$. Then
\begin{eqnarray}\label{eq:2.1}
\varphi_{12}(8\cdot3^{\beta})&=&
\sum_{d\,|\,8\cdot3^{\beta}}\mu\Big(\frac{8\cdot3^{\beta}}{d}\Big)\Big[\frac{d}{12}\Big]\nonumber\\
&=&\mu(1)\Big[\frac{8\cdot3^{\beta}}{12}\Big]+
\mu(2)\Big[\frac{4\cdot3^{\beta}}{12}\Big]
+\mu(3)\Big[\frac{8\cdot3^{\beta-1}}{12}\Big]+\mu(6)\Big[\frac{4\cdot3^{\beta-1}}{12}\Big]\nonumber\\
&=&2\cdot3^{\beta-1}-3^{\beta-1}
-2\cdot3^{\beta-2}+3^{\beta-2}\nonumber\\
&=&2\cdot3^{\beta-2}.\nonumber
\end{eqnarray}

(5) For the case $\alpha\geq4$, i.e.,$n=2^{\alpha}\cdot 3^{\beta}>12$, and so $\beta\geq0$. If $\beta=0$, i.e.,$n=2^{\alpha}(\alpha\geq4)$, then we have
\begin{eqnarray}\label{eq:2.1}
\varphi_{12}(2^{\alpha})&=&\sum_{d\,|\,2^{\alpha}}\mu\Big(\frac{2^{\alpha}}{d}\Big)\Big[\frac{d}{12}\Big]
=\Big[\frac{2^{\alpha-2}}{3}\Big]-\Big[\frac{2^{\alpha-3}}{3}\Big]\nonumber\\
&=& \frac{1}{3}\Big(2^{\alpha-2}-\frac{1}{2}\big(3-(-1)^{\alpha-2}\big)\Big)
-\frac{1}{3}\Big(2^{\alpha-3}-\frac{1}{2}\big(3-(-1)^{\alpha-3}\big)\Big)  \nonumber\\
&=&\frac{1}{3}\Big(2^{\alpha-3}+(-1)^{\alpha}\Big)\nonumber
\end{eqnarray}
If $\beta=1$, i.e.,$n=3\cdot2^{\alpha}$, then we have
\begin{eqnarray}\label{eq:2.1}
\varphi_{12}(3\cdot2^{\alpha})&=&\sum_{d\,|\,3\cdot2^{\alpha}}\mu\Big(\frac{3\cdot2^{\alpha}}{d}\Big)\Big[\frac{d}{12}\Big]
=2^{\alpha-2}-2^{\alpha-3}-\Big[\frac{2^{\alpha-2}}{3}\Big]+\Big[\frac{2^{\alpha-3}}{3}\Big]\nonumber\\
&=&\frac{1}{3}\Big(2^{\alpha-2}+(-1)^{\alpha+1}\Big).\nonumber
\end{eqnarray}
If $\beta\geq2$, we have
\begin{eqnarray}\label{eq:2.1}
\varphi_{12}(2^{\alpha}\cdot3^{\beta})&=&
\sum_{d\,|\,2^{\alpha}\cdot3^{\beta}}\mu\Big(\frac{2^{\alpha}\cdot3^{\beta}}{d}\Big)\Big[\frac{d}{12}\Big]\nonumber\\
&=&\mu(1)\Big[\frac{2^{\alpha}\cdot3^{\beta}}{12}\Big]+
\mu(2)\Big[\frac{2^{\alpha-1}\cdot3^{\beta}}{12}\Big]
+\mu(3)\Big[\frac{2^{\alpha}\cdot3^{\beta-1}}{12}\Big]+\mu(6)\Big[\frac{2^{\alpha-1}\cdot3^{\beta-1}}{12}\Big]\nonumber\\
&=&2^{\alpha-2}\cdot3^{\beta-1}-\Big[\frac{2^{\alpha-2}\cdot3^{\beta-1}}{2}\Big]
-2^{\alpha-2}\cdot3^{\beta-2}+\Big[\frac{2^{\alpha-2}\cdot3^{\beta-2}}{2}\Big]\nonumber\\
&=&2^{\alpha-2}\cdot3^{\beta-2}.\nonumber
\end{eqnarray}


This completes the proof of Theorem 4.1.
\end{proof}

Now consider the case $n=2^{\alpha}\,3^{\beta}n_{1}$, where $n_{1}>1$ and $\gcd(n_{1}, 6)=1$. We have the following

\begin{Thm}\label{thm:4.2}\quad Let $\alpha$ and $\beta$ be non-negative integers, $k, \alpha_{i}(1\leq i\leq k)$ be positive integers, and $p_{1},\ldots, p_k$ be distinct primes. Suppose that $\gcd(p_{i}, 6)=1$$(1\leq i\leq k)$ and $n=2^{\alpha}\,3^{\beta}\,\prod_{i=1}^{k} p_{i}^{\alpha_{i}}>12$, then
\begin{eqnarray}\label{eq:A3}
\;\;\;\;\varphi_{12}(n)=\begin{cases}
\frac{1}{12}\,\varphi(n)+\frac{1}{4}\,(-1)^{\Omega(n)}\cdot 2 \,^{\omega(n)-\alpha},\\
 \;\;\;\; \;\;\;\; \;\;\;\; \;\;\;\;\;\mbox{if $\alpha=0, 1$, $\beta=0$, and $R_{\mathbb{P}_{k}}'=\{7, 11\}, \{7\}$};\\
 \frac{1}{12}\,\varphi(n)+\frac{1}{4}\,(-1)^{\Omega(n)+1}\cdot 2 \,^{\omega(n)-\alpha},\\
 \;\;\;\; \;\;\;\; \;\;\;\; \;\;\;\;\;\mbox{if $\alpha=0, 1$, $\beta\geq2$,  and $R_{\mathbb{P}_{k}}'=\{7, 11\}, \{7\}, \{11\}$};\\
\frac{1}{12}\,\varphi(n)+\frac{1}{6}\,(-1)^{\Omega(n)+[\frac{\alpha+1}{2}]}\cdot 2 \,^{\omega(n)-[\frac{\alpha+1}{2}]-\beta}, \\
\;\;\;\; \;\;\;\; \;\;\;\; \;\;\;\;\; \mbox{if $\alpha=0, 1, 2$, $\beta=0, 1$,  and $R_{\mathbb{P}_{k}}'=\{5, 11\}, \{5\}$},\\
\;\;\;\; \;\;\;\; \;\;\;\; \;\;\;\;\; \mbox{or $\alpha=0, 1$, $\beta=1$,  and $R_{\mathbb{P}_{k}}'=\{11\}$};\\
\;\;\;\; \;\;\;\; \;\;\;\; \;\;\;\;\; \mbox{or $\alpha=2$, $\beta=0, 1$,  and  $R_{\mathbb{P}_{k}}'=\{11\}$};\\
\frac{1}{12}\,\varphi(n)+\frac{1}{6}\,(-1)^{\Omega(n)}\cdot 2 \,^{\omega(n)-\beta}, \\
\;\;\;\; \;\;\;\; \;\;\;\; \;\;\;\; \; \mbox{if $\alpha\geq3$, $\beta=0, 1$,  and $R_{\mathbb{P}_{k}}'=\{5, 11\}, \{5\}, \{11\}$};\\
\frac{1}{12}\,\varphi(n)+\frac{5}{12}\,(-1)^{\Omega(n)}\cdot2\,^{\omega(n)},\\
\;\;\;\; \;\;\;\; \;\;\;\; \;\;\;\;\; \mbox{if $\alpha=0$, $\beta=0$,  and $R_{\mathbb{P}_{k}}'=\{11\}$};\\
\frac{1}{12}\,\varphi(n)+\frac{1}{12}\,(-1)^{\Omega(n)}\cdot2\,^{\omega(n)-1},\\
\;\;\;\; \;\;\;\; \;\;\;\; \;\;\;\;\; \mbox{if $\alpha=1$, $\beta=0$,  and  $R_{\mathbb{P}_{k}}'=\{11\}$};\\
\frac{1}{12}\,\varphi(n), \;\;\;\;  \mbox{otherwise}.
 \end{cases}
\end{eqnarray}
\end{Thm}

\begin{proof}\quad  Set $n_1=\prod_{i=1}^{k} p_{i}^{\alpha_{i}}$, then $\gcd(n_{1}, 6) = 1$ and $n=2^{\alpha}\,3^{\beta}\,n_1$.

\textbf{Case 1}.  \quad For the case $\alpha= 0$.

\textbf{(A)} \quad If $\beta=0$, then $n_{1} > 1$. Thus by (1), (3)and Lemmas 2.1-2.2, we have
\begin{eqnarray}\label{eq:n/12}
\varphi_{12}(n)&= &\varphi_{12}(n_{1})=\sum_{d\,|\,n_{1}}\mu\Big(\frac{n_{1}}{d}\Big)\Big[\frac{d}{12}\Big]\\
&=&\frac{1}{12}\sum_{d\,|\,n_{1}}\mu\Big(\frac{n_{1}}{d}\Big)\Big(d-6+3\Big(\frac{-1}{d}\Big)+2\Big(\frac{-3}{d}\Big)\Big)\nonumber\\
&=&\frac{1}{12}\sum_{d\,|\,n_{1}}\mu\Big(\frac{n_{1}}{d}\Big)d-\frac{1}{2}\sum_{d\,|\,n_{1}}\mu\Big(\frac{n_{1}}{d}\Big)
+\frac{1}{4}\sum_{d\,|\,n_{1}}\mu\Big(\frac{n_{1}}{d}\Big)\Big(\frac{-1}{d}\Big)\nonumber\\
&&+\;\frac{1}{6}\sum_{d\,|\,n_{1}}\mu\Big(\frac{n_{1}}{d}\Big)\Big(\frac{-3}{d}\Big)\nonumber\\
&=&\frac{1}{12}\,\varphi(n_{1})+\frac{1}{4}\,\prod_{i=1}^{k}\Big(\Big(\frac{-1}{p_{i}}\Big)^{\alpha_{i}}-
\Big(\frac{-1}{p_{i}}\Big)^{\alpha_{i}-1}\Big)
\nonumber\\
&&+\,\frac{1}{6}\,\prod_{i=1}^{k}\Big(\Big(\frac{-3}{p_{i}}\Big)^{\alpha_{i}}-\Big(\frac{-3}{p_{i}}\Big)^{\alpha_{i}-1}\Big).\nonumber
\end{eqnarray}

If $1\in R_{\mathbb{P}_{k}}'$ or $\{5, 7\}\subseteq R_{\mathbb{P}_{k}}'$, then there exists  $p_{i}\equiv1 (\mathrm{mod}\, 12)$, or there exist $p_j$ and $p_l$ such that $p_{j}\equiv 5(\mathrm{mod}\, 12)$ and $p_{l}\equiv 7(\mathrm{mod}\, 12)$,  and then $(\frac{-1}{p_{i}}) = (\frac{-3}{p_{i}}) =1$, or $(\frac{-1}{p_{j}}) = (\frac{-3}{p_{l}}) =1$, respectively. Thus by (18) we can get
\begin{eqnarray}\label{eq:12n0}
\varphi_{12}(n)=\frac{1}{12}\,\varphi(n_{1})=\frac{1}{12}\,\varphi(n).
\end{eqnarray}

If $R_{\mathbb{P}_{k}}'=\{7, 11\}$ or $\{7\}$, i.e., for any $p\in\mathbb{P}_{k}$, we have $p\equiv7, 11(\mathrm{mod}\,12)$ or $p\equiv7(\mathrm{mod}\,12)$, respectively. This means that  $(\frac{-1}{p})=-1$, and there exists a prime $p'\equiv7(\mathrm{mod}\,12)$, i.e., $(\frac{-3}{p'})=1$, in either of the two cases. Thus by (18) we can obtain
\begin{eqnarray}\label{eq:12n1}
\varphi_{12}(n)=\frac{1}{12}\,\varphi(n_{1})+\frac{1}{4}\,\prod_{i=1}^{k}\Big(2\,(-1)^{\alpha_{i}}\Big)
=\frac{1}{12}\,\varphi(n)+\frac{1}{4}\,(-1)^{\Omega(n)}\,2^{\omega(n)}.
\end{eqnarray}

If $R_{\mathbb{P}_{k}}'=\{5, 11\}$ or $\{5\}$,  i.e., for any $p\in\mathbb{P}_{k}$,  $p\equiv5, 11(\mathrm{mod}\,12)$ or $p\equiv5(\mathrm{mod}\,12)$, respectively. Then $(\frac{-3}{p})=-1$, and there exists a prime $p'\equiv5(\mathrm{mod}\,12)$, i.e., $(\frac{-1}{p'})=1$ in either case. Thus by (18) we can get
\begin{eqnarray}\label{eq:12n2}
\varphi_{12}(n)=\frac{1}{12}\,\varphi(n_{1})+\frac{1}{6}\,\prod_{i=1}^{k}\Big(2\,(-1)^{\alpha_{i}}\Big)
=\frac{1}{12}\,\varphi(n)+\frac{1}{6}\,(-1)^{\Omega(n)}\,2^{\omega(n)}.
\end{eqnarray}

If $R_{\mathbb{P}_{k}}'=\{11\}$, i.e.,  for any $p\in \mathbb{P}_{k}$, $p\equiv11 (\mathrm{mod}\, 12)$, and then
 $(\frac{-1}{p})=(\frac{-3}{p})=-1$. Thus by (18) we have
\begin{eqnarray}\label{eq:12n3}
\;\;\;\varphi_{12}(n)=\frac{1}{12}\,\varphi(n_{1})+\frac{5}{12}\,\prod_{i=1}^{k}\Big(2\,(-1)^{\alpha_{i}}\Big)
=\frac{1}{12}\,\varphi(n)+\frac{5}{12}\,(-1)^{\Omega(n)}\,2\,^{\omega(n)}.
\end{eqnarray}

\textbf{(B)}\quad  If $\beta\geq1$,  then by (1) we have
\begin{eqnarray}\label{eq:12n00}
\varphi_{12}(n)
&=&\varphi_{12}(3^{\beta}n_{1})=\sum_{d\,|\,n_{1}}\mu\Big(\frac{3^{\beta}n_{1}}{d}\Big)\Big[\frac{d}{12}\Big]
+\sum_{d\,|\,3^{\beta-1}n_{1}}\mu\Big(\frac{3^{\beta}n_{1}}{3d}\Big)\Big[\frac{3d}{12}\Big]\nonumber\\
&=&\mu(3^{\beta})\,\varphi_{12}(n_{1})+\varphi_{4}(3^{\beta-1}n_{1}).\nonumber
\end{eqnarray}
Now from $\beta=1$, Lemma 2.4 and Case 1, we can get
\begin{eqnarray}\label{eq:12n01}
\varphi_{12}(n)&=&-\varphi_{12}(n_{1})+\varphi_{4}(n_{1})\\
&=&\begin{cases}
\frac{1}{12}\,\varphi(n),  & \mbox{if $R_{\mathbb{P}_{k}}'=\{7, 11\}, \{7\}$},\\
\frac{1}{12}\,\varphi(n)+\frac{1}{6}\,(-1)^{\Omega(n)}\,2 \,^{\omega(n)-1},  & \mbox{if $R_{\mathbb{P}_{k}}'=\{5, 11\}, \{5\}$},\\
\frac{1}{12}\,\varphi(n)+\frac{1}{6}\,(-1)^{\Omega(n)}\,2 \,^{\omega(n)-1},  & \mbox{if  $R_{\mathbb{P}_{k}}'=\{11\}$},\\
\frac{1}{12}\,\varphi(n),  & \mbox{otherwise}.
 \end{cases}\nonumber
\end{eqnarray}
For the case $\beta\geq2$, note that $\mu(3^{\gamma})=0$ with $\gamma\geq2$, thus by Lemma 2.4 we have
\begin{eqnarray}\label{eq:12n02}
\varphi_{12}(n)&=&\varphi_{4}(3^{\beta-1} n_{1})\\
&=&\begin{cases}
\frac{1}{12}\,\varphi(n)+\frac{1}{4}\,(-1)^{\Omega(n)+1}\,2 \,^{\omega(n)},  & \mbox{if $R_{\mathbb{P}_{k}}'=\{7, 11\}, \{7\}$},\\
\frac{1}{12}\,\varphi(n),  & \mbox{if $R_{\mathbb{P}_{k}}'=\{5, 11\}, \{5\}$},\\
\frac{1}{12}\,\varphi(n)+\frac{1}{4}\,(-1)^{\Omega(n)+1}\,2 \,^{\omega(n)},  & \mbox{if  $R_{\mathbb{P}_{k}}'=\{11\}$},\\
\frac{1}{12}\,\varphi(n),  & \mbox{otherwise}.
 \end{cases}\nonumber
\end{eqnarray}

From the above (18)-(24), Theorem 4.2 is proved in this case.

\textbf{Case 2}.\quad  For the case $\alpha= 1$.

 \textbf{(A)}\quad  If $\beta=0$, i.e., $n=2n_{1}$, then by (1), Case 1 and Lemma 2.4,  we have

 \begin{eqnarray}\label{eq:12n10}
\varphi_{12}(n)
&=&\sum_{d\,|\,n_{1}}\mu\Big(\frac{2n_{1}}{d}\Big)\Big[\frac{d}{12}\Big]+
\sum_{d\,|\,n_{1}}\mu\Big(\frac{2n_{1}}{2d}\Big)\Big[\frac{2d}{12}\Big]\\
&=&-\varphi_{12}(n_{1})+\varphi_{6}(n_{1})\nonumber\\
&=&\begin{cases}
\frac{1}{12}\,\varphi(n)+\frac{1}{4}\,(-1)^{\Omega(n)}\,2 \,^{\omega(n)-1},  & \mbox{if $R_{\mathbb{P}_{k}}'=\{7, 11\}, \{7\}$},\\
\frac{1}{12}\,\varphi(n)+\frac{1}{12}\,(-1)^{\Omega(n)+1}\,2 \,^{\omega(n)},  & \mbox{if $R_{\mathbb{P}_{k}}'=\{5, 11\}, \{5\}$},\\
\frac{1}{12}\,\varphi(n)+\frac{1}{12}\,(-1)^{\Omega(n)}\,2 \,^{\omega(n)-1},  & \mbox{if  $R_{\mathbb{P}_{k}}'=\{11\}$},\\
\frac{1}{12}\,\varphi(n),  & \mbox{otherwise}.
 \end{cases}\nonumber
\end{eqnarray}

\textbf{(B)}\quad If $\beta=1$, i.e., $n=6n_{1}$, then from (1) we can get
 \begin{eqnarray}\label{eq:12n11}
\varphi_{12}(n)
&=&\sum_{d\,|\,n_{1}}\mu\Big(\frac{6n_{1}}{d}\Big)\Big[\frac{d}{12}\Big]+
\sum_{d\,|\,n_{1}}\mu\Big(\frac{6n_{1}}{2d}\Big)\Big[\frac{2d}{12}\Big]+
\sum_{d\,|\,n_{1}}\mu\Big(\frac{6n_{1}}{3d}\Big)\Big[\frac{3d}{12}\Big]\nonumber\\
&& +
\sum_{d\,|\,n_{1}}\mu\Big(\frac{6n_{1}}{6d}\Big)\Big[\frac{6d}{12}\Big]\nonumber\\
&=&\varphi_{12}(n_{1})-\varphi_{6}(n_{1})-\varphi_{4}(n_{1})+\varphi_{2}(n_{1}).\nonumber
\end{eqnarray}
Now by Lemmas 2.4-2.5 and Case 1, we have
\begin{eqnarray}\label{eq:12n11}
\varphi_{12}(n)= \begin{cases}
\frac{1}{12}\,\varphi(n),  & \mbox{if $R_{\mathbb{P}_{k}}'=\{7, 11\}, \{7\}$},\\
\frac{1}{12}\,\varphi(n)+\frac{1}{12}\,(-1)^{\Omega(n)+1}\,2 \,^{\omega(n)-1},  & \mbox{if $R_{\mathbb{P}_{k}}'=\{5, 11\}, \{5\}$},\\
\frac{1}{12}\,\varphi(n)+\frac{1}{12}\,(-1)^{\Omega(n)+1}\,2 \,^{\omega(n)-1},  & \mbox{if  $R_{\mathbb{P}_{k}}'=\{11\}$},\\
\frac{1}{12}\,\varphi(n),  & \mbox{otherwise}.
 \end{cases}
\end{eqnarray}

\textbf{(C)}\quad If $\beta\geq 2$, then by (1) one can easily see that
 \begin{eqnarray}\label{eq:12n12}
\varphi_{12}(n)
&=&\sum_{d\,|\,n_{1}}\mu\Big(\frac{2\cdot3^{\beta}n_{1}}{d}\Big)\Big[\frac{d}{12}\Big]+
\sum_{d\,|\,n_{1}}\mu\Big(\frac{2\cdot3^{\beta}n_{1}}{2d}\Big)\Big[\frac{2d}{12}\Big]\nonumber\\
&&+\,\sum_{d\,|\,3^{\beta-1}n_{1}}\mu\Big(\frac{2\cdot3^{\beta} n_{1}}{3d}\Big)\Big[\frac{3d}{12}\Big]
+
\sum_{d\,|\,3^{\beta-1}n_{1}}\mu\Big(\frac{2\cdot3^{\beta}n_{1}}{6d}\Big)\Big[\frac{6d}{12}\Big]\nonumber\\
&=&-\varphi_{4}(3^{\beta-1}n_{1})+\varphi_{2}(3^{\beta-1}n_{1}).\nonumber
\end{eqnarray}
Now by Lemma 2.4 we can get
 \begin{eqnarray}\label{eq:12n12}
\varphi_{12}(n)
 =\begin{cases}
\frac{1}{12}\,\varphi(n)+\frac{1}{4}\,(-1)^{\Omega(n)+1}\,2 \,^{\omega(n)-1},  & \mbox{if $R_{\mathbb{P}_{k}}'=\{7, 11\}, \{7\}$},\\
\frac{1}{12}\,\varphi(n),  & \mbox{if $R_{\mathbb{P}_{k}}'=\{5, 11\}, \{5\}$},\\
\frac{1}{12}\,\varphi(n)+\frac{1}{4}\,(-1)^{\Omega(n)+1}\,2 \,^{\omega(n)-1},  & \mbox{if  $R_{\mathbb{P}_{k}}'=\{11\}$},\\
\frac{1}{12}\,\varphi(n),  & \mbox{otherwise}.
 \end{cases}
\end{eqnarray}

From the above (25)-(27), Theorem 4.2 is true in this case.

\textbf{Case 3}.\quad For the case $\alpha=2$.

 \textbf{(A)}\quad If $\beta=0$, i.e., $n=4n_{1}$, then from Lemmas 2.3 and 2.5, we can obtain
 \begin{eqnarray}\label{eq:12n20}
\varphi_{12}(n)&=&\sum_{d\,|\,4n_{1}}\mu\Big(\frac{4n_{1}}{d}\Big)\Big[\frac{d}{12}\Big]\\
&=&\sum_{d\,|\,n_{1}}\mu\Big(\frac{4n_{1}}{d}\Big)\Big[\frac{d}{12}\Big]+
\sum_{d\,|\,n_{1}}\mu\Big(\frac{4n_{1}}{2d}\Big)\Big[\frac{2d}{12}\Big]+
\sum_{d\,|\,n_{1}}\mu\Big(\frac{4n_{1}}{4d}\Big)\Big[\frac{4d}{12}\Big]\nonumber\\
&=&-\varphi_{6}(n_{1})+\varphi_{3}(n_{1})\nonumber\\
&=&\begin{cases}
\frac{1}{12}\,\varphi(n),  & \mbox{if $R_{\mathbb{P}_{k}}'=\{7, 11\}, \{7\}$},\\
\frac{1}{12}\,\varphi(n)+\frac{1}{12}\,(-1)^{\Omega(n)+1}\,2 \,^{\omega(n)},  & \mbox{if $R_{\mathbb{P}_{k}}'=\{5, 11\}, \{5\}$},\\
\frac{1}{12}\,\varphi(n)+\frac{1}{12}\,(-1)^{\Omega(n)+1}\,2 \,^{\omega(n)},  & \mbox{if  $R_{\mathbb{P}_{k}}'=\{11\}$},\\
\frac{1}{12}\,\varphi(n),  & \mbox{otherwise}.
 \end{cases}\nonumber
\end{eqnarray}

 \textbf{(B)}\quad If $\beta=1$, i.e., $n=12\,n_{1}$, then by the definition we have
 \begin{eqnarray}\label{eq:12n21}
\varphi_{12}(n)&=&\sum_{d|12n_{1}}\mu\Big(\frac{12n_{1}}{d}\Big)\Big[\frac{d}{12}\Big]\nonumber\\
&=&\sum_{d\,|\,n_{1}}\mu\Big(\frac{12 n_{1}}{d}\Big)\Big[\frac{d}{12}\Big]+
\sum_{d\,|\,n_{1}}\mu\Big(\frac{12 n_{1}}{2d}\Big)\Big[\frac{2d}{12}\Big]+
\sum_{d\,|\,n_{1}}\mu\Big(\frac{12 n_{1}}{4d}\Big)\Big[\frac{4d}{12}\Big]\nonumber\\
&& +
\sum_{d\,|\,n_{1}}\mu\Big(\frac{12 n_{1}}{3d}\Big)\Big[\frac{3d}{12}\Big]+
\sum_{d\,|\,n_{1}}\mu\Big(\frac{12 n_{1}}{6d}\Big)\Big[\frac{6d}{12}\Big]+
\sum_{d\,|\,n_{1}}\mu\Big(\frac{12 n_{1}}{12d}\Big)\Big[\frac{12 d}{12}\Big]\nonumber\\
&=&\varphi_{6}(n_{1})-\varphi_{3}(n_{1})-\varphi_{2}(n_{1})+\varphi(n_{1}).\nonumber
\end{eqnarray}
Now by Lemmas 2.3-2.4 and Case 1, we can get
\begin{eqnarray}\label{eq:12n21}
\varphi_{12}(n)= \begin{cases}
\frac{1}{12}\,\varphi(n),  & \mbox{if $R_{\mathbb{P}_{k}}'=\{7, 11\}, \{7\}$},\\
\frac{1}{12}\,\varphi(n)+\frac{1}{12}\,(-1)^{\Omega(n)+1}\,2 \,^{\omega(n)-1},  & \mbox{if $R_{\mathbb{P}_{k}}'=\{5, 11\}, \{5\}$},\\
\frac{1}{12}\,\varphi(n)+\frac{1}{12}\,(-1)^{\Omega(n)+1}\,2 \,^{\omega(n)-1},  & \mbox{if  $R_{\mathbb{P}_{k}}'=\{11\}$},\\
\frac{1}{12}\,\varphi(n),  & \mbox{otherwise}.
 \end{cases}
\end{eqnarray}

\textbf{(C)}\quad If $\beta\geq 2$, then from $n=4\cdot3^{\beta}\,n_{1}$ and the definition, we know that
 \begin{eqnarray}\label{eq:12n22}
\varphi_{12}(n)
&=&\sum_{d\,|\,4\cdot3^{\beta}n_{1}}\mu\Big(\frac{4\cdot3^{\beta}n_{1}}{d}\Big)\Big[\frac{d}{12}\Big]
=\sum_{d\,|\,2\cdot3^{\beta-1}n_{1}}\mu\Big(\frac{4\cdot3^{\beta}n_{1}}{6d}\Big)\Big[\frac{6d}{12}\Big]\nonumber\\
&=&\sum_{d\,|\,3^{\beta-1}n_{1}}\mu\Big(\frac{2\cdot3^{\beta}n_{1}}{2d}\Big)\Big[\frac{2d}{2}\Big]+
\sum_{d\,|\,3^{\beta-1}n_{1}}\mu\Big(\frac{2\cdot3^{\beta}n_{1}}{d}\Big)\Big[\frac{d}{2}\Big]\\
&=&\varphi(3^{\beta-1}n_{1})-\varphi_{2}(3^{\beta-1}n_{1})=\frac{1}{2}\varphi(3^{\beta-1}n_{1})\nonumber\\
&=&\frac{1}{12}\varphi(4\cdot3^{\beta}n_{1})=\frac{1}{12}\varphi(n).\nonumber
\end{eqnarray}

From the above (28)-(30), Theorem 4.2 is proved in this case.

\textbf{Case 4}. \quad For the case $\alpha\geq 3$.

\textbf{(A)}\quad If $\beta=0$, i.e., $n=2^{\alpha}\,n_{1}$,  then by Lemma 2.5 we have
 \begin{eqnarray}\label{eq:12n30}
\varphi_{12}(n)
&=&\sum_{d\,|\,n_{1}}\mu\Big(\frac{2^{\alpha}n_{1}}{d}\Big)\Big[\frac{d}{12}\Big]+
\sum_{d\,|\,2^{\alpha-1}n_{1}}\mu\Big(\frac{2^{\alpha}n_{1}}{2d}\Big)\Big[\frac{2d}{12}\Big]\\
&=&\varphi_{6}(2^{\alpha-1}n_{1})\nonumber\\
&=&\begin{cases}
\frac{1}{12}\,\varphi(n),  & \mbox{if $R_{\mathbb{P}_{k}}'=\{7, 11\}, \{7\}$},\\
\frac{1}{12}\,\varphi(n)+\frac{1}{6}\,(-1)^{\Omega(n)}\,2 \,^{\omega(n)},  & \mbox{if $R_{\mathbb{P}_{k}}'=\{5, 11\}, \{5\}$},\\
\frac{1}{12}\,\varphi(n)+\frac{1}{6}\,(-1)^{\Omega(n)}\,2 \,^{\omega(n)},  & \mbox{if  $R_{\mathbb{P}_{k}}'=\{11\}$},\\
\frac{1}{12}\,\varphi(n),  & \mbox{otherwise}.
 \end{cases}\nonumber
\end{eqnarray}

\textbf{(B)}\quad If $\beta=1$, i.e., $n=3\cdot2^{\alpha}\,n_{1}$, then by the definition we have
 \begin{eqnarray}\label{eq:12n31}
\varphi_{12}(n)
&=&\sum_{d\,|\,n_{1}}\mu\Big(\frac{3\cdot2^{\alpha} n_{1}}{d}\Big)\Big[\frac{d}{12}\Big]+
\sum_{d\,|\,2^{\alpha-1}n_{1}}\mu\Big(\frac{3\cdot2^{\alpha} n_{1}}{2d}\Big)\Big[\frac{2d}{12}\Big]\nonumber\\
&& +
\sum_{d\,|\,n_{1}}\mu\Big(\frac{3\cdot2^{\alpha} n_{1}}{3d}\Big)\Big[\frac{3d_{1}}{12}\Big]+
\sum_{d\,|\,2^{\alpha-1}n_{1}}\mu\Big(\frac{3\cdot2^{\alpha} n_{1}}{6d}\Big)\Big[\frac{6d}{12}\Big]\nonumber\\
&=&-\varphi_{6}(2^{\alpha-1}n_{1})+\varphi_{2}(2^{\alpha-1}n_{1})\nonumber\\
&=&\begin{cases}
\frac{1}{12}\,\varphi(n),  & \mbox{if $R_{\mathbb{P}_{k}}'=\{7, 11\}, \{7\}$},\\
\frac{1}{12}\,\varphi(n)+\frac{1}{12}\,(-1)^{\Omega(n)}\,2 \,^{\omega(n)},  & \mbox{if $R_{\mathbb{P}_{k}}'=\{5, 11\}, \{5\}$},\\
\frac{1}{12}\,\varphi(n)+\frac{1}{12}\,(-1)^{\Omega(n)}\,2 \,^{\omega(n)},  & \mbox{if  $R_{\mathbb{P}_{k}}'=\{11\}$},\\
\frac{1}{12}\,\varphi(n),  & \mbox{otherwise}.
 \end{cases}
\end{eqnarray}

\textbf{(C)}\quad If $\beta\geq 2$, then by Lemma 2.6 we can get
 \begin{eqnarray}\label{eq:12n32}
\varphi_{12}(n)=\frac{1}{12}\varphi(2^{\alpha}\cdot3^{\beta}n_{1})=\frac{1}{12}\varphi(n).
\end{eqnarray}

Now from (31)-(33), Theorem 4.2 is proved in this case.

From the above, we complete the proof for Theorem 4.2.
\end{proof}

\section{The Parity of The Generalized Euler Function $\varphi_{8}(n)$ and $\varphi_{12}(n)$}\label{sec:5}

Based on Theorem 3.1 and Theorems 4.1-4.2, this section gives the parity of  $\varphi_{8}(n)$ and $\varphi_{12}(n)$, respectively.

\begin{Thm}\label{thm:5.1}\quad If $n$ is a positive integer,  then $\varphi_{8}(n)$ is odd if and only if $n=8, 16$ or $n$ is given by the following table.
\begin{eqnarray*}
\begin{tabular}{|c |c| }\hline
$ n$ &\mbox{conditions}      \\ \hline
$p^{\alpha}$ & $p\equiv9, 15\,(\mathrm{mod}\, 16)$; $p\equiv3, 5(\mathrm{mod}\, 16)$, $2\,|\alpha$; $p\equiv11, 13(\mathrm{mod}\,16)$, $2\nmid\alpha$; \\ \hline
$2p^{\alpha}$ & $p\equiv7, 9\,(\mathrm{mod}\, 16)$; $p\equiv3, 13(\mathrm{mod}\, 16)$, $2\,|\alpha$; $p\equiv5, 11(\mathrm{mod}\,16)$, $2\nmid\alpha$;   \\ \hline
$4p^{\alpha}$ & $p\equiv3, 5\,(\mathrm{mod}\, 8)$;   \\ \hline
$8p^{\alpha}$ & $p\equiv3, 7\,(\mathrm{mod}\, 8)$;   \\ \hline
$p_{1}^{\alpha_{1}}p_{2}^{\alpha_{2}}$ & $p_{1}\equiv p_{2}\equiv3(\mathrm{mod}\, 8)$; $p_{1}\equiv p_{2}\equiv5(\mathrm{mod}\, 8)$; $p_{1}\equiv3(\mathrm{mod}\, 8), p_{2}\equiv5(\mathrm{mod}\, 8)$;  \\ \hline
$2p_{1}^{\alpha_{1}}p_{2}^{\alpha_{2}}$ & $p_{1}\equiv p_{2}\equiv3(\mathrm{mod}\, 8)$; $p_{1}\equiv p_{2}\equiv5(\mathrm{mod}\, 8)$; $p_{1}\equiv3(\mathrm{mod}\, 8), p_{2}\equiv5(\mathrm{mod}\, 8)$.   \\ \hline
\end{tabular} \nonumber
\end{eqnarray*}
Where $p, p_{1}, p_{2}$ are odd primes with $p_1\neq p_2$, and $\alpha, \alpha_{1},\alpha_{2}$ are positive integers.
\end{Thm}
\begin{proof}\quad For $n=2^{\alpha}$, by (6) we know that $\varphi_{8}(n)$ is odd if and only if $n=8, 16$.

Now suppose that $n=2^{\alpha}\,\prod_{i=1}^{k} p_{i}^{\alpha_{i}}$, where $\alpha \geq0$, $\alpha_{1},\ldots, \alpha_k$ are positive integers, and $p_{1},\ldots, p_k$ are distinct odd primes.
Set $n_{1}=\prod_{i=1}^{k} p_{i}^{\alpha_{i}}$, then $n_{1}>1$ is odd. By Theorem 3.1, we have the following four cases.

\textbf{Case 1}.\quad For the case $R_{\mathbb{P}_{k}}=\{5, 7\}$ or $\{5\}$.

\textbf{(A)}\quad If $\alpha=0$, i.e., $n=n_1$ is odd. By (7) we have $\varphi_{8}(n)=\frac{1}{8}\varphi(n)+\frac{1}{4}(-1)^{\Omega(n)}2 ^{\omega(n)}$. Note that there exists a prime factor $p$ of $n$ such that $p\equiv5 (\mathrm{mod}\, 8)$, and so we must have $\omega(n)\leq2$ if $\varphi_{8}(n)$ is odd. For $\omega(n)=2$, i.e., $n=p_{1}^{\alpha_{1}}p_{2}^{\alpha_{2}}$,  by (7) we have
\begin{eqnarray}\label{eq:2.1}
\varphi_{8}(n)=\frac{1}{8}\,p_{1}^{\alpha_{1}-1}(p_{1}-1)\,p_{2}^{\alpha_{2}-1}(p_{2}-1)+(-1)^{\alpha_{1}+\alpha_{2}}.\nonumber
\end{eqnarray}
Therefore $\varphi_{8}(n)$ is odd if and only if $p_{1}\equiv p_{2}\equiv5(\mathrm{mod}\, 8)$, which is true. Now for $\omega(n)=1$, i.e., $n=p_1^{\alpha_1}$ with $p_1\equiv 5(\mathrm{mod}\, 8)$, similarly, by (7) we have
\begin{eqnarray}\label{eq:2.1}
\varphi_{8}(n)=\frac{1}{8}\,p_1^{\alpha_{1}-1}(p_{1}-1)+\frac{1}{2}(-1)^{\alpha_{1}}
=\frac{1}{8}\,\big(p_{1}^{\alpha_{1}-1}(p_{1}-1)+4\cdot(-1)^{\alpha_{1}}\big).\nonumber
\end{eqnarray}
From $p_{1}\equiv5(\mathrm{mod}\, 8)$, we have $p_{1}\equiv5, 13(\mathrm{mod}\, 16)$. If $p_{1}\equiv5(\mathrm{mod}\, 16)$, then $$p_{1}^{\alpha_{1}-1}(p_{1}-1)+4(-1)^{\alpha_{1}}\equiv4\cdot5^{\alpha_{1}-1}+4(-1)^{\alpha_{1}}(\mathrm{mod}\, 16),$$ and so $\varphi_{8}(n)$ is odd if and only if $2\mid\alpha_{1}$. If $p_{1}\equiv13(\mathrm{mod}\, 16)$, then $$p_{1}^{\alpha_{1}-1}(p_{1}-1)+4(-1)^{\alpha_{1}}\equiv12\cdot(-3)^{\alpha_{1}-1}+4(-1)^{\alpha_{1}}(\mathrm{mod}\, 16),$$
thus $\varphi_{8}(n)$ is odd if and only if $\alpha_{1}$ is odd.

\textbf{(B)}\quad If $\alpha=1$, i.e., $\omega(n)\geq 2$, by (7) we have $\varphi_{8}(n)=\frac{1}{8}\varphi(n)+\frac{1}{4}(-1)^{\Omega(n)}2 ^{\omega(n)-1}$. Then we must have $\omega(n)\leq3$ if $\varphi_{8}(n)$ is odd. For $\omega(n)=3$, namely, $n=2p_{1}^{\alpha_{1}}p_{2}^{\alpha_{2}}$, using the same
method as (A), $\varphi_{8}(n)$ is odd if and only if $p_{1}\equiv p_{2}\equiv5(\mathrm{mod}\, 8)$. Now for $\omega(n)=2$, i.e., $n=2p_{1}^{\alpha_{1}}$ with $p_{1}\equiv5(\mathrm{mod}\, 8)$. Similar to (A), $\varphi_{8}(n)$ is odd if and only if $p_{1}\equiv5(\mathrm{mod}\, 16)$ and $\alpha_{1}$ is odd, or $p_{1}\equiv13(\mathrm{mod}\, 16)$ and $2\mid \alpha_{1}$ .

\textbf{(C)}\quad If $\alpha=2$, i.e., $\omega(n)\geq2$, then by (7) we have $\varphi_{8}(n)=\frac{1}{8}\varphi(n)=\frac{1}{4}\prod _{i=1}^{k} p_{i}^{\alpha_{i}-1}(p_{i}-1)$. Thus from the assumption $p_{i}\equiv 5, 7(\mathrm{mod}\, 8)$ or $p_{i}\equiv 5(\mathrm{mod}\, 8)$,  we know that $\omega(n)=2$ if $\varphi_{8}(n)$ is odd. In this case, $n=4p_{1}^{\alpha_{1}}$ with $p_{1}\equiv5(\mathrm{mod}\, 8)$, then
$p_{1}^{\alpha_{1}-1}(p_{1}-1)\equiv4(\mathrm{mod}\, 8)$, namely, $\varphi_{8}(n)$ is odd.

\textbf{(D)}\quad If $\alpha\geq3$, then by (7), $\varphi_{8}(n)=\frac{1}{8}\varphi(n)=2^{\alpha-4}\prod _{i=1} ^{k} p_{i}^{\alpha_{i}-1}(p_{i}-1)$. Thus we must have $\alpha=3$ and $k=1$ if
  $\varphi_{8}(n)$ is odd, namely, $n=8p_1^{\alpha_1}$ with $p_{1}\equiv 5(\mathrm{mod}\, 8)$. In this case, $\varphi_{8}(n)=\frac{1}{8}\varphi(n)=\frac{1}{2}p_{1}^{\alpha_{1}-1}(p_{1}-1)$ is always even.

\textbf{Case 2}. \quad For the case $R_{\mathbb{P}_{k}}=\{3, 7\}$ or $ \{3\}$.

\textbf{(A)}\quad If $\alpha=0$, by (7) we have $\varphi_{8}(n)=\frac{1}{8}\varphi(n)+\frac{1}{8}(-1)^{\Omega(n)}2 ^{\omega(n)}$. Thus we must have $\omega(n)\leq 3$ if $\varphi_{8}(n)$ is odd. For the case $\omega(n)=3$, i.e, $n=p_1^{\alpha_1}p_2^{\alpha_2}p_3^{\alpha_3}$, where $p_i\equiv 3(\mathrm{mod}\, 4) (i=1,2,3)$. Easily $\varphi_{8}(n)$ is always even in this case. Therefore we must have $\omega(n)=1,2$. For $\omega(n)=2$, i.e., $n=p_{1}^{\alpha_{1}}p_{2}^{\alpha_{2}}$. Note that $R_{\mathbb{P}_{k}}=\{3, 7\}$ or $ \{3\}$, then by (7),
$\varphi_{8}(n)=\frac{1}{8}\big(p_{1}^{\alpha_{1}-1}(p_{1}-1)\,p_{2}^{\alpha_{2}-1}(p_{2}-1)+4\cdot(-1)^{\alpha_{1}+\alpha_{2}}\big)$ is odd if and only if $p_1\equiv p_2\equiv3(\mathrm{mod}\, 8)$. Now for $\omega(n)=1$, i.e., $n=p_{1}^{\alpha_{1}}$ with $p_{1}\equiv3(\mathrm{mod}\, 8)$, then by (7) we have $\varphi_{8}(n)=\frac{1}{8}\big(p_{1}^{\alpha_{1}-1}(p_{1}-1)+2(-1)^{\alpha_{1}}\big).$ And so $\varphi_{8}(n)$ is odd if and only if
$p_{1}\equiv 3(\mathrm{mod}\, 16)$ and $2\mid\alpha_{1}$, or $p_{1}\equiv 11(\mathrm{mod}\, 16)$ and $\alpha_{1}$ is odd.

\textbf{(B)}\quad If $\alpha=1$, i.e., $\omega\geq2$, by (7) we have $\varphi_{8}(n)=\frac{1}{8}\varphi(n)+\frac{1}{8}(-1)^{\Omega(n)-1}2 ^{\omega(n)-1}$. Thus
we must have $\omega(n)\leq 3$ if $\varphi_{8}(n)$ is odd. Using the same method as (A) in case 1, we can get $\varphi_{8}(n)$ is odd if and only if $n=2p_{1}^{\alpha_{1}}p_{2}^{\alpha_{2}}$ with $p_{1}\equiv p_{2}\equiv3(\mathrm{mod}\, 8)$, or $n=2p_{1}^{\alpha_{1}}$ with $p_{1}\equiv3(\mathrm{mod}\, 16)$ and $2\mid \alpha_{1}$, or $p_{1}\equiv11(\mathrm{mod}\, 16)$ and $\alpha_{1}$ is odd.

\textbf{(C)}\quad If $\alpha=2$, i.e., $\omega(n)\geq2$, by (7) we have $\varphi_{8}(n)=\frac{1}{8}\varphi(n)+\frac{1}{8}\,(-1)^{\Omega(n)-1}\,2 \,^{\omega(n)}$.
Therefore we must have $\omega(n)\leq3$ if $\varphi_{8}(n)$ is odd. For the case $\omega(n)=3$, i.e., $n=4p_{1}^{\alpha_{1}}p_{2}^{\alpha_{2}}$,
we know that
$$\varphi_{8}(n)=\frac{1}{4}p_{1}^{\alpha_{1}-1}(p_1-1)p_{2}^{\alpha_{2}-1}(p_2-1)+(-1)^{\alpha_{1}+\alpha_{2}+1},$$ which is always even.
Now for the case $\omega(n)=2$, i.e., $n=4p_{1}^{\alpha_{1}}$ with $p_{1}\equiv 3(\mathrm{mod}\, 8)$, by (7) we have $\varphi_{8}(n)=\frac{1}{4}(p_{1}^{\alpha_{1}-1}(p_1-1)+2(-1)^{\alpha_{1}+1})$. Since $$p_{1}^{\alpha_{1}-1}(p_1-1)+2(-1)^{\alpha_{1}+1}\equiv2\cdot3^{\alpha_{1}-1}+2(-1)^{\alpha_{1}+1}\equiv 4(\mathrm{mod}\, 8), $$
 and so $\varphi_{8}(n)$ is odd.

\textbf{(D)}\quad If $\alpha\geq3$, by (7) we have $\varphi_{8}(n)=\frac{1}{8}\varphi(n)=2^{\alpha-4}\prod _{i=1} ^{k} p_{i}^{\alpha_{i}-1}(p_{i}-1)$.
From $R_{\mathbb{P}_{k}}=\{3, 7\}$ or $ \{3\}$, we must have $\alpha=3$ and $k=1$ if $\varphi_{8}(n)$ is odd, namely, $n=8p_{1}^{\alpha_{1}}$ with  $p_{1}\equiv3(\mathrm{mod}\, 8)$. Obviously,
$\varphi_{8}(n)= \frac{1}{2}\,p_{1}^{\alpha_{1}-1}(p_{1}-1)$ is odd in this case.

\textbf{Case 3}.\quad For the case $R_{\mathbb{P}_{k}}=\{7\}$.

\textbf{(A)}\quad  If $\alpha=0$, by (7), $\varphi_{8}(n)=\frac{1}{8}\varphi(n)+\frac{3}{8}\,(-1)^{\Omega(n)}\,2 \,^{\omega(n)}$.
 Then we must have $\omega(n)\leq2$ if $\varphi_{8}(n)$ is odd. For $\omega(n)=2$, i.e., $n=p_{1}^{\alpha_{1}}p_{2}^{\alpha_{2}}$, and then
$$
\varphi_{8}(n)=\frac{1}{2}\Big(p_{1}^{\alpha_{1}-1}p_{2}^{\alpha_{2}-1}\cdot\frac{p_{1}-1}{2}\cdot\frac{p_{2}-1}{2}
+3\cdot(-1)^{\alpha_{1}+\alpha_{2}}\Big).
$$
Since $p_{1}\equiv p_{2}\equiv7(\mathrm{mod}\, 8)$, we have $\frac{p_{1}-1}{2}\cdot\frac{p_{2}-1}{2}\equiv1(\mathrm{mod}\, 4)$
and
$$p_{1}^{\alpha_{1}-1}p_{2}^{\alpha_{2}-1}\cdot\frac{p_{1}-1}{2}\cdot\frac{p_{2}-1}{2}
+3\cdot(-1)^{\alpha_{1}+\alpha_{2}}
\equiv  (-1)^{\alpha_{1}+\alpha_{2}-2}+3\cdot(-1)^{\alpha_{1}+\alpha_{2}}
\equiv  0(\mathrm{mod}\, 4),$$
which means that $\varphi_{8}(n)$ is even. Now for $\omega(n)=1$, i.e., $n=p_{1}^{\alpha_{1}}$, by (7) we have $$\varphi_{8}(n)=\frac{1}{4}\Big(p_{1}^{\alpha_{1}-1}\cdot\frac{p_{1}-1}{2}+3\cdot(-1)^{\alpha_{1}}\Big).$$ Now from $p_{1}\equiv7(\mathrm{mod}\, 8)$, we have
 $p_{1}\equiv7, 15(\mathrm{mod}\, 16)$.
If  $p_{1}\equiv7(\mathrm{mod}\, 16)$,  then $$p_{1}^{\alpha_{1}-1}\cdot\frac{p_{1}-1}{2}+3\cdot(-1)^{\alpha_{1}}\equiv3\cdot(-1)^{\alpha_{1}-1}+3\cdot(-1)^{\alpha_{1}}\equiv0(\mathrm{mod}\, 8), $$  namely, $\varphi_{8}(n)$ is even. And so $p_{1}\equiv15(\mathrm{mod}\, 16)$, then $$p_{1}^{\alpha_{1}-1}\cdot\frac{p_{1}-1}{2}+3\cdot(-1)^{\alpha_{1}}\equiv7\cdot(-1)^{\alpha_{1}-1}+3\cdot(-1)^{\alpha_{1}}\equiv4(\mathrm{mod}\, 8), $$  namely, $\varphi_{8}(n)$ is odd.

\textbf{(B)}\quad  If  $\alpha=1$, by (7), $\varphi_{8}(n)=\frac{1}{8}\varphi(n)+\frac{1}{8}(-1)^{\Omega(n)}2\, ^{\omega(n)-1}$. Using the similar proof for (A) in case 1, $\varphi_{8}(n)$ is odd if and only if $n=2p_{1}^{\alpha_{1}}$ and $p_{1}\equiv7(\mathrm{mod}\, 16)$.

\textbf{(C)}\quad  If  $\alpha=2$, i.e., $\omega(n)\geq2$, by (7), $\varphi_{8}(n)=\frac{1}{8}\varphi(n)+\frac{1}{8}\,(-1)^{\Omega(n)-1}\,2 \,^{\omega(n)}$.
Then we must have $\omega(n)\leq3$ if $\varphi_{8}(n)$ is odd.
For $\omega(n)=3$, i.e., $n=4p_{1}^{\alpha_{1}}p_{2}^{\alpha_{2}}$ with $p_{1}\equiv p_{2}\equiv7(\mathrm{mod}\, 8)$, we know that
$$\varphi_{8}(n)=p_{1}^{\alpha_{1}-1}p_{2}^{\alpha_{2}-1}\cdot\frac{p_{1}-1}{2}\cdot\frac{p_{2}-1}{2}
+(-1)^{\alpha_{1}+\alpha_{2}-1}$$
 is always even. Now for $\omega(n)=2$, i.e., $n=4p_{1}^{\alpha_{1}}$ with $p_{1}\equiv7(\mathrm{mod}\, 8)$, we can verify that
 $$\varphi_{8}(n)=\frac{1}{2}\Big(p_{1}^{\alpha_{1}-1}\cdot\frac{p_{1}-1}{2}+(-1)^{\alpha_{1}-1}\Big)$$ is also even.

\textbf{(D)}\quad If  $\alpha\geq3$, by (7), $\varphi_{8}(n)=\frac{1}{8}\varphi(n)=2^{\alpha-4}\prod _{i=1} ^{k} p_{i}^{\alpha_{i}-1}(p_{i}-1)$. Hence, by $R_{\mathbb{P}_{k}}=\{7\}$ we know that $\varphi_{8}(n)$ is odd if and only if $\alpha=3$ and $k=1$, i.e., $n=8p_{1}^{\alpha_{1}}$ with $p_{1}\equiv7(\mathrm{mod}\, 8)$.

\textbf{Case 4}.\quad For the case $\{3, 5\}\subseteq R_{\mathbb{P}_{k}}$ or $1\in R_{\mathbb{P}_{k}}$.

\textbf{(A)}\quad If $\{3, 5\}\subseteq R_{\mathbb{P}_{k}}$, i.e., $k\geq 2$, then by (7) we have $\varphi_{8}(n)=\frac{1}{8}\varphi(n)=\frac{1}{8}\varphi(2^{\alpha})\prod _{i=1} ^{k}p_{i}^{\alpha_{i}-1}$ $(p_{i}-1)$.  Thus we must have $k=2$ and $\alpha\leq1$ if $\varphi_{8}(n)$ is odd, namely, $n=p_{1}^{\alpha_{1}}p_{2}^{\alpha_{2}}$ or $2p_{1}^{\alpha_{1}}p_{2}^{\alpha_{2}}$, where $p_{1}\equiv3(\mathrm{mod}\, 8)$ and $p_{2}\equiv5(\mathrm{mod}\, 8)$. Obviously,  $\varphi_{8}(n)=\frac{1}{8}p_{1}^{\alpha_{1}-1}(p_{1}-1)p_{2}^{\alpha_{2}-1}(p_{2}-1)$ is always odd in this case.

\textbf{(B)}\quad If $1\in R_{\mathbb{P}_{k}}$, by (7), $\varphi_{8}(n)=\frac{1}{8}\varphi(n)=\frac{1}{8}\varphi(2^{\alpha})\prod _{i=1} ^{k} p_{i}^{\alpha_{i}-1}(p_{i}-1)$.
Thus, we must have $\alpha\leq1$ and $k=1$ if  $\varphi_{8}(n)$ is odd. Namely, $n=p_{1}^{\alpha_{1}}, 2p_{1}^{\alpha_{1}}$ with $p_{1}\equiv1(\mathrm{mod}\, 8)$, and then $\varphi_{8}(n)=\frac{1}{8}p^{\alpha_{1}-1}(p_{1}-1)$. Obviously, $\varphi_{8}(n)$ is odd if and only if $p_{1}\equiv9(\mathrm{mod}\, 16)$.

From the above, we complete the proof of Theorem 5.1.
\end{proof}

\begin{Thm}\label{thm:5.2} If $n$ is a positive integer, then $\varphi_{12}(n)$ is odd if and only if $n=2^{\alpha}\,(\alpha\geq4)$,
$3\cdot2^{\alpha}\,(\alpha\geq2)$, $2\cdot3^{\beta}\,(\beta\geq2)$,  $4\cdot3^{\beta}\,(\beta\geq2)$, or satisfies the follow table.
\begin{eqnarray*}
\begin{tabular}{|c |c| }\hline
$ n$ &\mbox{conditions}      \\ \hline
$p^{\alpha}$ & $p\equiv13, 17, 19, 23\,(\mathrm{mod}\, 24)$;  \\ \hline
$2p^{\alpha}$ & $p\equiv7, 11, 13, 17\,(\mathrm{mod}\, 24)$;  \\ \hline
$3p^{\alpha}$ & $p\equiv5, 7\,(\mathrm{mod}\, 12)$;   \\ \hline
$4p^{\alpha}$ & $p\equiv5, 7\,(\mathrm{mod}\, 12)$;   \\ \hline
$6p^{\alpha}$ & $p\equiv5, 7\,(\mathrm{mod}\, 12)$;   \\ \hline
$12p^{\alpha}$ & $p\equiv5, 11\,(\mathrm{mod}\, 12)$.  \\ \hline
\end{tabular} \nonumber
\end{eqnarray*}
Where $p>3$ is an odd prime and $\alpha\geq 1$.
\end{Thm}
\begin{proof}
Obviously, by the definition of $\varphi_{12}(n)$ we can get that $\varphi_{12}(n)=0$ for $n<12$ and $\varphi_{12}(n)=1$  for $n=12, 24$, and then we consider about $n>12$ and $n\neq 24$.  First, for the case $n=2^{\alpha}\cdot3^{\beta}$.

If $\alpha=0$, we have $\beta\geq3$, then by (16), $\varphi_{12}(n)=\frac{1}{2}\big(3^{\beta-2}-(-1)^{\beta}\big)$ is even.

If $\alpha=1$, we have $\beta\geq2$, then by (16), $\varphi_{12}(n)=\frac{1}{2}\big(3^{\beta-2}-(-1)^{\beta+1}\big)$ is odd.

If $\alpha=2$, we have $\beta\geq2$, then by (16), $\varphi_{12}(n)=3^{\beta-2}$ is odd.

If $\alpha=3$, we have $\beta\geq2$, then by (16), $\varphi_{12}(8\cdot3^{\beta})=2\cdot3^{\beta-2}$ is even.

If $\alpha\geq4$, then $\beta\geq0$. For $\beta=0$, by (16), $\varphi_{12}(n)=\frac{1}{3}(2^{\alpha-3}+(-1)^{\alpha})$ is odd. For $\beta=1$,
 by (16), $\varphi_{12}(n)=\frac{1}{3}\big(2^{\alpha-2}+(-1)^{\alpha+1}\big)$ is odd. For $\beta\geq2$,  by (16),
 $\varphi_{12}(2^{\alpha}\cdot3^{\beta})=2^{\alpha-2}\cdot3^{\beta-2},$ which is always
 even.

Next, we consider the case $n=2^{\alpha}\,3^{\beta}n_{1}$, where $\alpha \geq0$, $\beta \geq0$, $n_{1}>1$ and $\gcd(n_{1}, 6)=1$.  For convenience, we set   $n=2^{\alpha}\,3^{\beta}\prod_{i=1}^{k} p_{i}^{\alpha_{i}}$,
where $\alpha_{i} \geq1$,  $p_{i}$ is an odd prime and $p_{i}>3(1\leq i\leq k)$.
By Theorem 4.2 we have the following four cases.

\textbf{Case 1}.\quad $R_{\mathbb{P}_{k}}'=\{7, 11\}$ or $\{7\}$.

\textbf{(A)}\quad $\alpha=0$.  If $\beta=0$, i.e., $\omega(n)\geq 1$, from (17) we have $\varphi_{12}(n)=\frac{1}{12}\varphi(n)+\frac{1}{4}\,(-1)^{\Omega(n)}\,2 \,^{\omega(n)}$.
Thus, from the assumption $R_{\mathbb{P}_{k}}'=\{7, 11\}\ \textmd{or}\ \{7\}$, we must have $\omega(n)\leq 2$ if $\varphi_{12}(n)$ is odd. For $\omega(n)= 2$, i.e., $n=p_{1}^{\alpha_{1}}p_{2}^{\alpha_{2}}$, note that $R_{\mathbb{P}_{k}}'=\{7, 11\}\ \textmd{or}\ \{7\}$, and then $$\varphi_{12}(n)=\frac{1}{3}\,p_{1}^{\alpha_{1}-1}\,p_{2}^{\alpha_{2}-1}\cdot\frac{p_{1}-1}{2}\cdot\frac{p_{2}-1}{2}+(-1)^{\alpha_{1}
+\alpha_{2}}$$
is always even in this case. And so $\omega(n)=1$, i.e., $n=p_{1}^{\alpha_{1}}$, by (20),
$$\varphi_{12}(n)=\frac{1}{12}\big(p_{1}^{\alpha_{1}-1}(p_{1}-1)
+6\cdot(-1)^{\alpha_{1}}\big).$$
Note that $p_{1}\equiv7(\mathrm{mod}\,12)$, i.e., $p_{1}\equiv7, 19(\mathrm{mod}\,24)$. If $p_{1}\equiv7(\mathrm{mod}\,24)$, then $p_{1}^{\alpha_{1}-1}(p-1)+6\cdot(-1)^{\alpha_{1}}\equiv0(\mathrm{mod}\, 24)$, which means that $\varphi_{12}(n)$ is even. And so $p_{1}\equiv19(\mathrm{mod}\,24)$, thus $p_{1}^{\alpha_{1}-1}(p_{1}-1)+6\cdot(-1)^{\alpha_{1}}\equiv12(\mathrm{mod}\, 24)$, namely, $\varphi_{12}(n)$ is odd.

 If $\beta=1$, i.e., $\omega(n)\geq 2$, by (17),  $\varphi_{12}(n)=\frac{1}{12}\varphi(n)=\frac{1}{6}\prod _{i=1} ^{k} p_{i}^{\alpha_{i}-1}(p_{i}-1)$. Similarly, we must have $\omega(n)=2$ if $\varphi_{12}(n)$ is odd. In this case, $n=3p_{1}^{\alpha}$ with $p_{1}\equiv7(\mathrm{mod}\, 12)$,
 then $\varphi_{12}(n)=\frac{1}{6}p_{1}^{\alpha-1}(p_{1}-1)$, easy to see that $\varphi_{12}(n)$ is always even.
 If $\beta\geq2$, i.e., $\omega(n)\geq 2$, by (17), $\varphi_{12}(n)=\frac{1}{12}\varphi(n)+\frac{1}{4}\,(-1)^{\Omega(n)+1}\,2 \,^{\omega(n)}$. Similarly, we must have $\omega(n)=2$, i.e., $n=3^{\beta}p_{1}^{\alpha_{1}}$ if $\varphi_{12}(n)$ is odd. Since $p_{1}\equiv7(\mathrm{mod}\, 12)$, and so  $\varphi_{12}(n)=3^{\beta-2}p_{1}^{\alpha_{1}-1}\cdot\frac{p_{1}-1}{2}+(-1)^{\beta+\alpha_{1}+1}$ is always even.

\textbf{(B)}\quad $\alpha=1$. If $\beta=0$, i.e., $\omega(n)\geq2$, by (25), $\varphi_{12}(n)=\frac{1}{12}\varphi(n)+\frac{1}{4}(-1)^{\Omega(n)}2^{\omega(n)-1}$. Similarly, we must have $\omega(n)\leq3$ if $\varphi_{12}(n)$ is odd. For $\omega(n)=3$, i.e., $n=2p_{1}^{\alpha_{1}}p_{2}^{\alpha_{2}}$, note that $R_{\mathbb{P}_{k}}'=\{7, 11\}\ \textmd{or}\ \{7\}$, and then easy to see that
$\varphi_{12}(n)$ is always even.
And so $\omega(n)= 2$, i.e.,  $n=2p_{1}^{\alpha_{1}}$, note that $p_{1}\equiv7(\mathrm{mod}\, 12)$, namely, $p_{1}\equiv7, 19(\mathrm{mod}\, 24)$. In this case,  $\varphi_{12}(n)$ is odd if and only if $p_{1}\equiv7(\mathrm{mod}\, 24)$.

 If $\beta=1$, i.e., $\omega(n)\geq3$, by (17), $\varphi_{12}(n)=\frac{1}{12}\varphi(n)=\frac{1}{6}\prod _{i=1} ^{k} p_{i}^{\alpha_{i}-1}(p_{i}-1)$. Similarly, from $R_{\mathbb{P}_{k}}'=\{7, 11\}\ \textmd{or}\ \{7\}$, we can get $\varphi_{12}(n)$ is odd if and only if $\omega(n)=3$, i.e., $n=6p_{1}^{\alpha}$ with $p_{1}\equiv7(\mathrm{mod}\, 12)$.

 If $\beta\geq2$, i.e., $\omega(n)\geq3$, by(17),  $\varphi_{12}(n)=\frac{1}{12}\varphi(n)+\frac{1}{4}(-1)^{\Omega(n)+1}2 ^{\omega(n)-1}$. Then we must have $\omega(n)=3$
if $\varphi_{12}(n)$ is odd, namely, $n=2\cdot3^{\beta}p_{1}^{\alpha_{1}}$ with $p_{1}\equiv7(\mathrm{mod}\, 12)$. Obviously,  $\varphi_{12}(n)=3^{\beta-2}p^{\alpha-1}\cdot\frac{p-1}{2}+(-1)^{2+\beta+\alpha}$ is always even in this case.

\textbf{(C)}\quad $\alpha=2$. If $\beta=0$, i.e., $\omega(n)\geq 2$, by (17),  $\varphi_{12}(n)=\frac{1}{12}\varphi(n)=\frac{1}{6}\prod _{i=1} ^{k} p_{i}^{\alpha_{i}-1}(p_{i}-1)$. Thus, we must have $\omega(n)=2$ if $\varphi_{12}(n)$ is odd. In this case, $n=4p_{1}^{\alpha_{1}}$ with $p_{1}\equiv7(\mathrm{mod}\, 12)$, then $p_1^{\alpha_1}(p_1-1)\equiv 6(\mathrm{mod}\, 12)$, which means that $\varphi_{12}(n)$ is always odd in this case.

 If $\beta\geq1$, i.e., $\omega(n)\geq 3$, by (17), $\varphi_{12}(n)=\frac{1}{12}\varphi(n)= 3^{\beta-2}\prod _{i=1} ^{k} p_{i}^{\alpha_{i}-1}(p_{i}-1)$.  Note that $R_{\mathbb{P}_{k}}'=\{7, 11\}\ \textmd{or}\ \{7\}$, and then $\varphi_{12}(n)$ is always even.

\textbf{(D)}\quad $\alpha\geq3$. By (17) and $R_{\mathbb{P}_{k}}'=\{7, 11\}\ \textmd{or}\ \{7\}$, $\varphi_{12}(n)=\frac{1}{12}\varphi(n)=\frac{1}{3}\cdot2^{\alpha-3}\varphi(3^{\beta}n_{1})$ is always even in this case.

\textbf{Case 2}.\quad $R_{\mathbb{P}_{k}}'=\{5, 11\}$ or $\{5\}$.

\textbf{(A)}\quad $\alpha=0$. If $\beta=0$, i.e., $\omega(n)\geq1$, from (21),  we can get  $\varphi_{12}(n)=\frac{1}{12}\,\varphi(n)+\frac{1}{6}\,(-1)^{\Omega(n)}\,2^{\omega(n)}$. Thus, we must have $\omega(n)=1$ if $\varphi_{12}(n)$ is odd, namely, $n=p_{1}^{\alpha_{1}}$ with $p_{1}\equiv5(\mathrm{mod}\, 12)$. Hence
$$\varphi_{12}(p_{1}^{\alpha_{1}})=\frac{1}{12}\,p_{1}^{\alpha_{1}-1}(p_{1}-1)+\frac{1}{3}\,(-1)^{\alpha_{1}}
=\frac{1}{3}\Big(p_{1}^{\alpha_{1}-1}\cdot\frac{p_{1}-1}{4}+(-1)^{\alpha_{1}}\Big). $$
Note that $p_{1}\equiv5(\mathrm{mod}\, 12)$, i.e., $p_{1}\equiv5,17(\mathrm{mod}\, 24)$. If $p_{1}\equiv5(\mathrm{mod}\, 24)$, then $p_{1}^{\alpha-1}\cdot\frac{p_{1}-1}{4}+(-1)^{\alpha_{1}}\equiv0(\mathrm{mod}\, 6)$, which means $\varphi_{12}(p_{1}^{\alpha})$ is always even. And so $p_{1}\equiv17(\mathrm{mod}\, 24)$, in this case $p_{1}^{\alpha_{1}-1}\cdot\frac{p_{1}-1}{4}+(-1)^{\alpha_{1}}\equiv3(\mathrm{mod}\, 6)$, namely, $\varphi_{12}(p^{\alpha_{1}})$ is odd.

 If $\beta=1$, i.e., $\omega(n)\geq2$, from (23) we have $\varphi_{12}(n)=\frac{1}{12}\,\varphi(n)+\frac{1}{6}\,(-1)^{\Omega(n)}\,2^{\omega(n)-1}$. Thus, we must have
 $\omega(n)=2$ if $\varphi_{12}(n)$ is odd, namely, $n=3p_{1}^{\alpha_{1}}$ with $p_{1}\equiv5(\mathrm{mod}\, 12)$, in this case
$$\varphi_{12}(3p_{1}^{\alpha_{1}})
=\frac{1}{3}\Big(2\,p_{1}^{\alpha_{1}-1}\cdot\frac{p_{1}-1}{4}+(-1)^{\alpha_{1}}\Big)$$
is always odd.

 If $\beta\geq2$, i.e., $\omega(n)\geq2$, from (24) we can get $\varphi_{12}(n)=\frac{1}{12}\,\varphi(n)$. We must have $\omega(n)=2$, i.e., $n=3^{\beta}p^{\alpha}(\beta\geq2)$, if $\varphi_{12}(n)$ is odd. From the assumption $p_{1}\equiv5(\mathrm{mod}\, 12)$, $\varphi_{12}(3^{\beta}\,p_{1}^{\alpha_{1}})=\frac{1}{12}\,\varphi(3^{\beta}\,p_{1}^{\alpha_{1}})
=2\cdot 3^{\beta-1}\,p_{1}^{\alpha_{1}-1}\cdot\frac{p_{1}-1}{4}$ is always even.

\textbf{(B)}\quad $\alpha=1$, i.e., $\omega(n)\geq 2$. By (25)-(27), then we must have $\omega(n)\leq3$ if $\varphi_{12}(n)$ is odd.
Namely, $n=2p_{1}^{\alpha_{1}}, 2p_{1}^{\alpha_{1}}p_{2}^{\alpha_{2}}, 6p_{1}^{\alpha_{1}}$, or $2\cdot3^{\beta}p_{1}^{\alpha_{1}}(\beta\geq2)$. Similar to the proof of (A) in case 1, $\varphi_{12}(n)$ is odd if and only if $n=2p_{1}^{\alpha_{1}}$  with $p_{1}\equiv17(\mathrm{mod}\, 24)$, or $n=6p_{1}^{\alpha_{1}}$  with  $p_{1}\equiv5(\mathrm{mod}\, 12)$.

\textbf{(C)}\quad $\alpha=2$. If $\beta=0$, i.e., $\omega(n)\geq2$, by (28), $\varphi_{12}(n)=\frac{1}{12}\varphi(n)+\frac{1}{12}(-1)^{\Omega(n)+1}2^{\omega(n)}$,
then we must have $\omega(n)=2$ if $\varphi_{12}(n)$ is odd. In this case, $n=4p_{1}^{\alpha_{1}}$ with $p_{1}\equiv5(\mathrm{mod}\, 12)$. Hence,
$\varphi_{12}(n)=\frac{1}{6}\,p_{1}^{\alpha_{1}-1}(p_{1}-1)+\frac{1}{3}\,(-1)^{\alpha_{1}+3}
=\frac{1}{3}\Big(p_{1}^{\alpha_{1}-1}\frac{p_{1}-1}{2}+(-1)^{\alpha_{1}+3}\Big)$ is always odd.

If $\beta=1$, i.e., $\omega(n)\geq3$, by (29), $\varphi_{12}(n)=\frac{1}{12}\varphi(n)+\frac{1}{12}(-1)^{\Omega(n)+1}2^{\omega(n)-1}$, we must have $\omega(n)=3$ if $\varphi_{12}(n)$ is odd. In this case, $n=12 \,p_{1}^{\alpha_{1}}$ with $p_{1}\equiv5(\mathrm{mod}\, 12)$, and then
$\varphi_{12}(n)=\frac{1}{3}\,p_{1}^{\alpha_{1}-1}(p_{1}-1)+\frac{1}{3}\,(-1)^{\alpha_{1}+4}
=\frac{1}{3}\big(p_{1}^{\alpha_{1}-1}(p_{1}-1)+(-1)^{\alpha_{1}+4}\big)$ is odd.

If $\beta\geq 2$, i.e., $\omega(n)\geq3$, by (30), $\varphi_{12}(n)=\frac{1}{12}\varphi(n)$, we must have $\omega(n)=3$ if $\varphi_{12}(n)$ is odd. Namely, $n=4\cdot3^{\beta} p_{1}^{\alpha_{1}}$ with $p_{1}\equiv5(\mathrm{mod}\, 12)$, and then
$$\varphi_{12}(n)=\frac{1}{12}\,\varphi(n)=3^{\beta-2}p_{1}^{\alpha_{1}-1}(p_{1}-1)$$ is always even.

\textbf{(D)}\quad  $\alpha\geq3$, i.e., $\omega(n)\geq 2$.
If $\beta=0$, then by (31) and $R_{\mathbb{P}_{k}}'=\{5, 11\}\ \textmd{or}\ \{5\}$, we konw that $\varphi_{12}(n)=\frac{1}{12}\varphi(n)+\frac{1}{6}(-1)^{\Omega(n)}\,2 \,^{\omega(n)}$ is always even in this case.

If $\beta=1$. By (32) and $R_{\mathbb{P}_{k}}'=\{5, 11\}\ \textmd{or}\ \{5\}$, we know that $\varphi_{12}(n)=\frac{1}{12}\varphi(n)+\frac{1}{12}(-1)^{\Omega(n)}\,2 \,^{\omega(n)}$ is always even in this case.

If $\beta\geq2$. By (33) and $R_{\mathbb{P}_{k}}'=\{5, 11\}\ \textmd{or}\ \{5\}$, $\varphi_{12}(n)=\frac{1}{12}\varphi(n)$ is always even in this case.

\textbf{Case 3}.\quad $R_{\mathbb{P}_{k}}'=\{11\}$.

\textbf{(A)}\quad  $\alpha=0$, i.e., $\omega(n)\geq 1$. From (22)-(24), we must have $\omega(n)\leq2$ if $\varphi_{12}(n)$ is odd.
For $\omega(n)=2$, i.e., $n=p_{1}^{\alpha_{1}}p_{2}^{\alpha_{2}}$, or $3^{\beta} p_{1}^{\alpha_{1}}(\beta\geq1)$ with $p_{1}\equiv p_{2}\equiv11(\mathrm{mod}\, 12)$. Thus, by (22)-(24), $\varphi_{12}(n)$ is  always even.
Hence, $\omega(n)=1$,  i.e., $n=p_{1}^{\alpha_{1}}$ with $p\equiv11(\mathrm{mod}\, 12)$, then from (22) we can get
$$\varphi_{12}(p_{1}^{\alpha_{1}})=\frac{1}{12}\,p_{1}^{\alpha_{1}-1}(p_{1}-1)+\frac{5}{6}\,(-1)^{\alpha_{1}}
=\frac{1}{6}\Big(p_{1}^{\alpha_{1}-1}\cdot\frac{p_{1}-1}{2}+5(-1)^{\alpha_{1}}\Big).$$
Note that $p_{1}\equiv11(\mathrm{mod}\, 12)$, i.e., $p_{1}\equiv11, 23(\mathrm{mod}\, 24)$. If $p_{1}\equiv11(\mathrm{mod}\, 24)$, then
 $$p_{1}^{\alpha_{1}-1}\cdot\frac{p_{1}-1}{2}+5(-1)^{\alpha_{1}}
\equiv5(-1)^{\alpha_{1}-1}+5(-1)^{\alpha_{1}}\equiv0(\mathrm{mod}\, 12),$$
namely, $\varphi_{12}(n)$  is even.
If $p_{1}\equiv23(\mathrm{mod}\, 24)$, then
 $$p_{1}^{\alpha_{1}-1}\cdot\frac{p_{1}-1}{2}+5(-1)^{\alpha_{1}}
\equiv11(-1)^{\alpha_{1}-1}+5(-1)^{\alpha_{1}}\equiv6(\mathrm{mod}\, 12),$$
namely, $\varphi_{12}(n)$  is odd.

\textbf{(B)}\quad $\alpha=1$, i.e., $\omega(n)\geq 2$. From (25)-(27), we must have $\omega(n)\leq3$ if $\varphi_{12}(n)$ is odd.
Namely, $n= 2p_{1}^{\alpha_{1}}$, $2p_{1}^{\alpha_{1}}p_{2}^{\alpha_{2}}$, $6p_{1}^{\alpha_{1}}$, or $2\cdot3^{\beta} p_{1}^{\alpha_{1}}(\beta\geq2)$ with $p_{1}\equiv p_{2}\equiv11(\mathrm{mod}\, 12)$. Using the same method as (A) in case 1, $\varphi_{12}(n)$ is odd if and only if $n=2p_{1}^{\alpha_{1}}$ with $p_{1}\equiv11(\mathrm{mod}\,24)$.

\textbf{(C)}\quad $\alpha=2$, i.e., $\omega(n)\geq 2$. If $\beta=0$, by (28), we must have $\omega(n)=2$ if $\varphi_{12}(n)$ is odd, namely, $n=4p_{1}^{\alpha_{1}}$ with $ p_{1}\equiv11(\mathrm{mod}\, 12)$.
Then by (28), $$\varphi_{12}(4p_{1}^{\alpha_{1}})=\frac{1}{3}\Big(p_{1}^{\alpha_{1}-1}\frac{p_{1}-1}{2}+(-1)^{\alpha_{1}+3}\Big)$$ is always even.

If $\beta\geq1$, i.e., $\omega(n)\geq3$, by (29)-(30),  we must have $\omega(n)=3$ if $\varphi_{12}(n)$ is odd.
Namely, $n=4\cdot3^{\beta} p_{1}^{\alpha_{1}}(\beta\geq 1)$ with
$p_{1}\equiv11(\mathrm{mod}\, 12)$.
 If $\beta\geq2$, then by(30),
$$\varphi_{12}(4\cdot3^{\beta} p_{1}^{\alpha_{1}})=\frac{1}{12}\varphi(4\cdot3^{\beta} p_{1}^{\alpha_{1}})=3^{\beta-2}\,p_{1}^{\alpha_{1}-1}(p_{1}-1)$$
is always even.  And so $\beta=1$, by (29),
$\varphi_{12}(12\,p_{1}^{\alpha_{1}})=\frac{1}{3}\big(p_{1}^{\alpha_{1}-1}(p_{1}-1)+(-1)^{\alpha_{1}+3}\big)$ is odd.

\textbf{(D)}\quad $\alpha\geq3$ . If $\beta=0$, i.e., $\omega(n)\geq 2$, then by (31) and $R_{\mathbb{P}_{k}}'=\{11\}$, we know that $\varphi_{12}(n)=\frac{1}{12}\varphi(n)+\frac{1}{6}(-1)^{\Omega(n)}\,2 \,^{\omega(n)}$ is always even in this case.

If $\beta=1$, i.e., $\omega(n)\geq 3$, then by (32) and $R_{\mathbb{P}_{k}}'=\{11\}$, we know that $\varphi_{12}(n)=\frac{1}{12}\varphi(n)+\frac{1}{12}(-1)^{\Omega(n)}\,2 \,^{\omega(n)}$ is always even in this case.

If $\beta\geq2$, i.e., $\omega(n)\geq 3$, then by (33) and $R_{\mathbb{P}_{k}}'=\{11\}$, we know that $\varphi_{12}(n)=\frac{1}{12}\varphi(n)
=2^{\alpha-2}\cdot3^{\beta-1}\prod_{i=1}^{k} p_{i}^{\alpha_{i}-1}( p_{i}-1)$ is always even.

\textbf{Case 4}.\quad $\{5, 7\}\subseteq R_{\mathbb{P}_{k}}'$ or $1\in R_{\mathbb{P}_{k}}'$.

\textbf{(A)}\quad If $\{5, 7\}\subseteq  R_{\mathbb{P}_{k}}'$, by (17) we have $\varphi_{12}(n)=\frac{1}{12}\varphi(n)$ is always even.

\textbf{(B)}\quad  If $1\in R_{\mathbb{P}_{k}}'$, then by (17), $\varphi_{12}(n)=\frac{1}{12}\varphi(n)$, thus we must have $k=1$, $\alpha\leq1$, and $\beta=0$ if $\varphi_{12}(n)$ is odd. Namely, $n=p_{1}^{\alpha_{1}}$ or $2p_{1}^{\alpha_{1}}$ with $p_{1}\equiv1(\mathrm{mod}\, 12)$.
In this case, $\varphi_{12}(n)=\frac{1}{12}\,p_{1}^{\alpha_{1}-1}(p_{1}-1)$ is odd if and only if $p_{1}\equiv13(\mathrm{mod}\, 24)$.

From the above, we complete the proof of Theorem 5.2.
\end{proof}

\section{Final Remark}\label{sec:6}

In [4] and [7], Cai,  et al.  gave the explicit  formulae  of the generalized Euler functions $\varphi_{e}(n)$ for $e=3, 4, 6$.
The key point is that the computing for $[\frac{n}{e}]$ can depend on the computing of the corresponding Jacobi symbol for $e=3,4,6$.
In the present paper, based on Lemmas 2.1-2.2, the exact formulae of $\varphi_8(n)$ and $\varphi_{12}(n)$ are given and then the parity is determined.

\begin{Conj}\label{conj6.1}
Let $e>1$ be a given integer. For any integer $d>2$ with $\gcd(d, e)=1$, there exist $u\in\mathbb{Q}$,  $a_{1}, a_{2}, a_{3}$, $b_{i}\,(1\leq j\leq r) \in \mathbb{Z}$, and $q_{j}\,(1\leq j\leq r)\in \mathbb{P}$, such that
\begin{eqnarray}\label{eq:d/e1}
\;\;\;\;\;\;\Big[\frac{d}{e}\Big]=u\Big(a_{1}d+a_{2}+a_{3}\Big(\frac{-1}{d}\Big)+\sum_{j=1}^{r}b_{j}\Big(\frac{\varepsilon_{j}\, q_{j}}{d}\Big)\Big)\;\mbox{$($$d$  odd$)$},
\end{eqnarray}
or
\begin{eqnarray}\label{eq:d/e2}
\Big[\frac{d}{e}\Big]=u\Big(a_{1}d+a_{2}+\sum_{j=1}^{r}b_{j}\Big(\frac{\varepsilon_{j}\,d}{q_{j}}\Big)\Big)
\;\mbox{$($$d$  even$)$}, \end{eqnarray}
where   $r\geq 1$ and $\varepsilon_{j}\in\{1, -1\}$.
\end{Conj}

Easy to see that the conjecture 6.1 is true for $e=2,3,4,6,8$ and $12$.
(see [4], [10] and (2), (3)).

\section{Acknowledgements}

The present paper is supported by Natural Science Foundation of China with No.11861001 and No.12071321, the Applied Basic Research Projects for Sichuan Province with No.2018JY0458, and the Construction Plan for Scientific Research Innovation Team of Provincial Colleges and Universities in Sichuan with No.18TD0047.


\begin{thebibliography}{20}

\bibitem{Shaghay2016}A. Al-Shaghay, K. Dilcher, Analogues of the binomial coefficient theorems of Gauss and Jacobi, Int. J.
Number Theory, 2016, 12(8): 2125-2145.

\bibitem{Cai2002}T. X. Cai, X. D. Fu, and X.  Zhou, A congruence involving the quotients of Euler and its applications (I),
Acta Arith., 2002, 103(4): 313-320.

\bibitem{Cai2007}T. X. Cai, X. D. Fu, and X.  Zhou, A congruence involving the quotients of Euler and its applications (II),
Acta Arith., 2007, 130(3): 203-214.

\bibitem{Cai2013} T. X. Cai, Z. Y. Shen  and M. J. Hu, On the parity of the generalized Euler function (I), Adv. Math. (China),
2013, 42(4): 505-510.

\bibitem{Cai2019} T. X. Cai, H. Zhong, S. Chern,  Congruence Involving the Quotients of Euler and Its Applications (III),
Acta Math. Sinica, Chinese Series,  2019, 62(4): 529-540.

\bibitem{Lehmer1938} E. Lehmer, On congruences involving Bernoulli numbers and the quotients of Fermat
and Wilson, Ann. of Math., 1938,  39: 350-359.

\bibitem{Liao2018} Q. Y. Liao and W. L. Luo, The computing formula for two classes of  generalized Euler functions, J. of Math.(PRC), 2018, 39(1): 97-110.

\bibitem{Liao:2019} Q. Y. Liao, The explicit formula for a special class of  generalized Euler functions, J. Sichuan Normal Univ.(Nat. Ed.), 2019, 42(3): 354-357.

\bibitem{Ribenboim1979} P. Ribenboim, 13 Lectures on Fermat¡¯s Last Theorem, Springer, 1979.


\bibitem{Shen2016} Z. Y. Shen, T. X. Cai, and M. J. Hu, On the parity of the generalized Euler function (II), Adv. Math. (China),
2016, 45(4): 509-519.

\bibitem{Wang2018} R. Wang and Q. Y. Liao, On the generalized Euler function $\varphi_{5}(n)$, J. Sichuan Normal Univ.(Nat. Ed.), 2018, 41(4): 445-449.



\end{thebibliography}
\end{document}